\documentclass[11pt]{amsart}
\usepackage{amssymb}
\usepackage{graphicx}
\usepackage{url}
\usepackage{tikz}
\usetikzlibrary{patterns}

\newcommand\restr[2]{{
  \left.\kern-\nulldelimiterspace
  #1
  \vphantom{\big|}
  \right|_{#2}
  }}
\newcommand{\suchthat}{\;\ifnum\currentgrouptype=16 \middle\fi|\;}

\newcommand\bint[2]{{
  \left.\kern-\nulldelimiterspace
  \vphantom{\int}
  \right.^b \hspace{-0.25cm}\int_{#1}^{#2}
  }}

\newcommand\pgfmathsinandcos[3]{
  \pgfmathsetmacro#1{sin(#3)}
  \pgfmathsetmacro#2{cos(#3)}
}
\newcommand\LatitudePlane[4][current plane]{
  \pgfmathsinandcos\sinEl\cosEl{#2}
  \pgfmathsinandcos\sint\cost{#3}
  \pgfmathsetmacro\yshift{\cosEl*\sint}
  \tikzset{#1/.estyle={cm={\cost,0,0,\cost*\sinEl,(#4,\yshift)}}}
}
\newcommand\DrawLatitudeCircle[3][1]{
  \LatitudePlane{\angEl}{#2}{#3}
  \tikzset{current plane/.prefix style={scale=#1}}
  \pgfmathsetmacro\sinVis{sin(#2)/cos(#2)*sin(\angEl)/cos(\angEl)}

  \pgfmathsetmacro\angVis{asin(min(1,max(\sinVis,-1)))}
  \draw[current plane, thick, color = red] (\angVis:1) arc (\angVis:-\angVis-180:1);
  \draw[current plane,dashed, thick, color = red] (180-\angVis:1) arc (180-\angVis:\angVis:1);
}

\newcommand\DrawGenus[7]{
  \pgfmathsetmacro{\xstart}{#1 - (0.985*#4)}
  \pgfmathsetmacro{\ystart}{#2 + (0.2*#3)}

	\draw[color = #6, rotate around={#5:(#1,#2)}, #7] (\xstart, \ystart) arc (190:350:#4  and #3);
	\draw[color = #6, rotate around={#5:(#1,#2)}, #7] (\xstart, \ystart) arc (190:210:#4  and #3) arc (150:30:#4  and #3) arc (330:350:#4  and #3);
}

\newcommand\DrawDonut[7]{
  \pgfmathsetmacro{\fctr}{.08}
  \pgfmathsetmacro{\newwidth}{0.5*#4}
  \pgfmathsetmacro{\newheight}{0.5*#3}
  \draw[color = #6, rotate around={#5:(#1,#2)}, #7] (#1, #2) ellipse (#4  and #3);
  \DrawGenus{#1}{#2}{\newheight}{\newwidth}{#5}{#6}{#7}
}

\newcommand\DrawTwoDonut[7]{
  \pgfmathsetmacro{\fctr}{.08}
  \pgfmathsetmacro{\newwidth}{0.25*#4}
  \pgfmathsetmacro{\newheight}{0.25*#3}
  \draw[color = #6, rotate around={#5:(#1,#2)}, #7] (#1, #2) ellipse (#4  and #3);
  \pgfmathsetmacro{\xcoordstart}{#1 - (#4*.4)*cos(#5) - (0.05*#3)*sin(#5)}
  \pgfmathsetmacro{\ycoordstart}{#2 - (#4*.4)*sin(#5) - (0.05*#3)*cos(#5)}
  \DrawGenus{\xcoordstart}{\ycoordstart}{\newheight}{\newwidth}{#5}{#6}{#7}
  \pgfmathsetmacro{\xcoordstart}{#1 + (#4*.4)*cos(#5) - (0.05*#3)*sin(#5)}
  \pgfmathsetmacro{\ycoordstart}{#2 + (#4*.4)*sin(#5) - (0.05*#3)*cos(#5)}
  \DrawGenus{\xcoordstart}{\ycoordstart}{\newheight}{\newwidth}{#5}{#6}{#7}
}

\usepackage{enumerate}

\usepackage[latin1]{inputenc}

\newtheorem{theorem}{Theorem}
\newtheorem{proposition}[theorem]{Proposition}
\newtheorem{lemma}[theorem]{Lemma}
\newtheorem{corollary}[theorem]{Corollary}

\theoremstyle{definition}
\newtheorem{definition}[theorem]{Definition}
\newtheorem{remark}[theorem]{Remark}
\newtheorem{example}[theorem]{Example}
\newtheorem{claim}[theorem]{Claim}

\newcommand{\CC}{{\mathbb C}}

\newcommand{\SSS}{{\mathbb S}}
\newcommand{\TT}{{\mathbb T}}

\newcommand{\cL}{{\mathcal L}}
\newcommand{\del}{\partial}

\begin{document}

\title{Toric actions on $b$-symplectic manifolds}

\author{Victor Guillemin}\address{Department of Mathematics, MIT, Cambridge MA, USA} \email{vwg@math.mit.edu}
\author{Eva Miranda}\address{Departament de Matem\`{a}tica Aplicada I, Universitat Polit\`{e}cnica de Catalunya, Barcelona, Spain}\email{eva.miranda@upc.edu}\thanks{Eva Miranda is partially supported by Ministerio de Econom\'{i}a y Competitividad project GEOMETRIA ALGEBRAICA, SIMPLECTICA, ARITMETICA Y APLICACIONES with reference: MTM2012-38122-C03-01/FEDER and by the European Science Foundation network CAST}
\author{Ana Rita Pires} \address{Department of Mathematics, Cornell University, Ithaca NY, USA}\email{apires@math.cornell.edu}\thanks{Ana Rita Pires had the partial support of an AMS-Simons Travel Grant.}
\author{Geoffrey Scott}\address{African Institute for Mathematical Sciences, Mbour, Senegal} \email{gsscott@aims-senegal.org}\thanks{Geoffrey Scott was partially supported by NSF RTG grant DMS-1045119 and by NSF RTG grant DMS-0943832.}

\date{\today}

\begin{abstract} We study Hamiltonian actions on $b$-symplectic manifolds with a focus on the effective case of half the dimension of the manifold. In particular, we prove a Delzant-type theorem that classifies these manifolds using polytopes that reside in a certain enlarged and decorated version of the dual of the Lie algebra of the torus.
\end{abstract}

\maketitle

\section{Introduction}

It is a well-known fact that the image of the moment map of a compact symplectic toric manifold is a Delzant polytope and that the symplectic manifold can be reconstructed from this polytope \cite{guilleminsternberg, guilleminsternberg2,delzant}.

In this paper we prove a similar theorem for a class of Poisson manifolds which is close to the symplectic class called \emph{$b$-symplectic manifolds}. These objects were first studied as manifolds with boundary in the works of Melrose \cite{book:melrose} and Nest and Tsygan \cite{nest}; recent treatments of the subject in \cite{guimipi11} and \cite{guimipi12} study the objects as manifolds with distinguished hypersurfaces. The symplectic groupoids integrating $b$-manifolds have been lately considered in \cite{gualtierili} and the topology of these manifolds has been studied further in \cite{cavalcanti}, \cite{frejlichmartinezmiranda} and \cite{marcutosorno1, marcutosorno2}. The $b$-symplectic manifolds that we consider are compact and have the property that the induced symplectic foliation on the exceptional hypersurface has compact leaves, the exceptional hypersurface in this case is a \emph{symplectic mapping torus} (see \cite{guimipi11}).

To define the moment map of a torus action on a $b$-manifold, we first enlarge the codomain $\mathfrak{t}^*$ to include points ``at infinity.'' The preimage of these points will be the exceptional hypersurface of the $b$-symplectic manifold. We will also assign $\mathfrak{t}^*$-valued weights to these points to encode certain geometric data, called the \emph{modular periods} of the components of the exceptional hypersurface. The definition of a Delzant polytope generalizes in a natural way to this enlarged codomain, giving the definition of a \emph{Delzant $b$-polytope}. The main theorem of this paper states that there is a bijection between $b$-symplectic toric manifolds and Delzant $b$-polytopes.

In contrast with classic symplectic geometry, the topology of the codomain of the moment map will depend on the $b$-manifold itself in two ways. First, the smooth structure on the codomain will depend on the modular periods of the exceptional hypersurfaces of the $b$-manifold. Second, in some cases the codomain will be contractible, and in other cases it will be topologically a circle.

This Delzant theorem allows us to classify all $2n$-dimensional $b$-symplectic toric manifolds into two categories. The first kind of $b$-symplectic toric manifold has as its underlying manifold $X_{\Delta} \times \mathbb{T}^2$, where $X_{\Delta}$ is any classic $(2n-2)$-dimensional symplectic toric manifold. The second kind of $b$-symplectic toric manifold is constructed from the manifold $X_{\Delta} \times \mathbb{S}^2$ by a sequence of symplectic cuts performed away from the exceptional hypersurface.

\section{Preliminary definitions and examples}

\subsection{$b$-objects, including $b$-functions}

We begin by recalling some of the notions introduced in detail in \cite{guimipi12}. A \textbf{$b$-manifold} is a pair $(M,Z)$ consisting of an oriented smooth manifold $M$ and a closed embedded hypersurface $Z$. A $b$-map $(M, Z) \rightarrow (M', Z')$ is an orientation-preserving map $f: M \rightarrow M'$ such that $f^{-1}(Z') = Z$ and $f$ is transverse to $Z'$. The sections of the \textbf{$b$-tangent bundle}, $^b TM$, are the vector fields on $M$ which at points of $Z$ are tangent to $Z$. The dual to this bundle is $^b T^*M$, the \textbf{$b$-cotangent bundle}. The sections of $\Lambda^k(^b T^*M)$ are called \textbf{$b$-de Rham $k$-forms} or simply \textbf{$b$-forms}. The space of all such forms is written $^b\Omega^k(M)$. The restriction of any $b$-form to $M \backslash Z$ is a classic differential form on $M \backslash Z$, and there is a differential $d: {^b}\Omega^k(M) \rightarrow {^b}\Omega^{k+1}(M)$ that extends the classic differential on $M \backslash Z$. With respect to this differential, we extend the standard definitions of closed and exact differential forms to \textbf{closed} $b$-forms and \textbf{exact} $b$-forms. A \textbf{$b$-symplectic} form is a closed $b$-form of degree 2 that has maximal rank (as a section of $\Lambda^2(^bT^*M)$) at every point of $M$. A $b$-symplectic manifold consists of the data of a $b$-manifold $(M, Z)$ and a $b$-symplectic form $\omega$. A \textbf{$b$-symplectomorphism} between two $b$-symplectic manifolds $(M,\omega)$ and $(M',\omega')$ is a $b$-map $\varphi:M\to M'$ such that $\varphi^*\omega'=\omega$.

Although a $b$-form can be thought of as a differential form with a singularity along $Z$, the singularity is so tame that it is possible to define the integral of a form of top degree by taking its principal value near $Z$.
\begin{definition}\label{def:lv}
For any $b$-form $\eta \in {^b}\Omega^{n}(M)$ on a $n$-dimensional $b$-manifold and any local defining function $y$ of $Z$, the \textbf{Liouville Volume} of $\eta$ is
\[
\bint{M}{}\eta := \lim_{\varepsilon \rightarrow 0} \int_{M \backslash \{-\varepsilon \leq y \leq \varepsilon\}} \eta
\]
\end{definition}
The fact that the limit in Definition \ref{def:lv} exists and is independent of $y$ is explained in \cite{radko} (for surfaces) and \cite{scott} (in the general case).

In \cite{guimipi12}, the authors prove that every $b$-form $\eta \in {^b}\Omega^p(M)$ can be written in a neighborhood of $Z = \{y = 0\}$ as
\[
\eta = \frac{dy}{y} \wedge \alpha + \beta
\]
for smooth forms $\alpha \in \Omega^{p-1}(M)$ and $\beta \in \Omega^p(M)$. Although the forms $\alpha$ and $\beta$ in this expression are not unique, the pullback $i_Z^*(\alpha)$ is unique, where $i_Z$ is the inclusion $Z \subseteq M$. The resulting differential form on $Z$ admits an alternative description: if $v$ is a vector field on $M$ such that $\restr{dy(v)}{Z} = 1$, then the vector field $\mathbb{L}:= yv$ is a $b$-vector field, $\restr{\mathbb{L}}{Z}$ doesn't depend on $v$ or $y$, the $b$-form $\iota_{\mathbb{L}}\eta$ is a smooth form, and $i_Z^*(\alpha) = i_Z^*\iota_{\mathbb{L}}\eta$. For this reason, we adopt the notation $\iota_{\mathbb{L}}\eta$ for this $(p-1)$-form on $Z$.

One can also study $b$-symplectic manifolds from the perspective of Poisson geometry: the dual of a $b$-symplectic form is a Poisson bivector whose top exterior product vanishes transversely (as a section of $\Lambda^{2n}(TM)$) at $Z$. Using these tools, we learn that $Z$ has a codimension-one symplectic foliation. To study this foliated hypersurface, we review the definition of the \textbf{modular vector field} on $M$.

\begin{definition}
Fix a volume form $\Omega$ on a $b$-symplectic manifold. The \textbf{modular vector field} $v^\Omega_\text{mod}$ on $M$ (or simply $v_{\textrm{mod}}$ if $\Omega$ is clear from the context) is the vector field defined by the derivation
\[
f\mapsto\frac{\mathcal{L}_{u_f}\Omega}{\Omega},
\]
where $u_f$ is the Hamiltonian vector field of $f$, defined by $df = \iota_{u_f}\omega$.
\end{definition}
Although the modular vector field depends on $\Omega$, different choices of $\Omega$ yield modular vector fields that differ by Hamiltonian vector fields. On a $b$-symplectic manifold, the modular vector field is tangent to the exceptional hypersurface $Z$ and its flow preserves the symplectic foliation of $Z$, and Hamiltonian vector fields are tangent to the symplectic foliation. In fact, in \cite{guimipi12} it is shown that corresponding to each modular vector field $v_{\textrm{mod}}$ and compact leaf $\mathcal{L}$ of a component $Z'$ of $Z$, there is a $k \in \mathbb{R}_{>0}$ and a symplectomorphism $f: \mathcal{L} \rightarrow \mathcal{L}$ such that $Z'$ is the mapping torus
\[
\frac{\mathcal{L} \times [0, k] }{(\ell, 0) \sim (f(\ell), k)}
\]
and the time-$t$ flow of $v_{\textrm{mod}}$ is translation by $t$ in the second coordinate. The number $k$, which depends only on the choice of component $Z' \subseteq Z$, is called the \textbf{modular period} of $Z'$. This definition generalizes the one given in \cite{radko} for $b$-symplectic surfaces.

Not all closed $b$-forms on a $b$-manifold are locally exact. For example, if $y$ is a local defining function for $Z$, then $\frac{dy}{y}$ is closed, but it is not exact in any neighborhood of any point of $Z$. Poincar\'{e}'s lemma is such a fundamental property of the (smooth) de Rham complex that we are motivated to enlarge the sheaf $C^{\infty}$ on a $b$-manifold to include functions such as $\log|y|$ so that we have a Poincar\'{e} lemma in $b$-geometry.

\begin{definition}
Let $(M, Z)$ be a $b$-manifold. The sheaf \footnote{If no global defining function for $Z$ exists (for example, if $Z$ is a meridian of $\mathbb{T}^2$), then this definition yields only a presheaf and ${^b}C^{\infty}$ is defined as its sheafification.} ${^b}C^{\infty}$ is defined by
\[
{^b}C^{\infty}(U) := \left\{c\log|y| + f \suchthat \begin{array}{l}c \in \mathbb{R}\\ y \ \textrm{is any defining function for} \ U \cap Z \subseteq U \\ f \in C^{\infty}(U)\end{array} \right\}
\]
Global sections of ${^b}C^{\infty}$ are called \textbf{$b$-functions}.
\end{definition}

Replacing $C^{\infty}$ with $^bC^{\infty}$ also enlarges the possible Hamiltonian torus actions on $b$-manifolds. In fact, in Corollary \ref{cor:notsmooth} we show that unless $Z = \emptyset$ there are no examples of effective Hamiltonian $\mathbb{T}^n$-actions on $2n$-dimensional $b$-symplectic manifolds with all their Hamiltonians in $C^{\infty}(M)$. We prove a simple relationship between the modular period and $b$-functions which will be useful in later sections.

\begin{proposition}\label{prop:propertiesofmodularperiod}
Let $(M, Z, \omega)$ be a $b$-symplectic manifold such that $Z$ is connected and has modular period $k$. Let $\pi: Z \rightarrow \SSS^1 \cong \mathbb{R}/k$ be the projection to the base of the corresponding mapping torus. Let $\gamma:\SSS^1 = \mathbb{R}/k \rightarrow Z$ be any loop with the property that $\pi\circ \gamma$ is the positively-oriented loop of constant velocity 1. The following numbers are equal:
\begin{itemize}
\item The modular period of $Z$.
\item $\int_{\gamma} \iota_{\mathbb{L}}\omega$.
\item The value $-c$ for any ${^b}C^{\infty}$ function $H = c\log|y|+f$ such that the Hamiltonian $X_H$ has 1-periodic orbits homotopic in $Z$ to some $\gamma$.
\end{itemize}
\begin{proof}
Recall from \cite{guimipi12} that $\iota_{\mathbb{L}}\omega (v_{\textrm{mod}})$ is the constant function 1. Let $s: [0,k] \rightarrow Z$ be a trajectory of the modular vector field. Because the modular period is $k$, $s(0)$ and $s(k)$ are in the same leaf $\mathcal{L}$ of the foliation. Let $\hat{s}: [0, k+1]\rightarrow Z$ be a smooth extension of $s$ such that $\restr{s}{[k, k+1]}$ is a path in $\mathcal{L}$ joining $\hat{s}(k) = s(k)$ to $\hat{s}(k+1) = s(0)$, making $\hat{s}$ a loop. Then
\[
k = \int_0^k 1 dt = \int_s \iota_{\mathbb{L}}\omega = \int_{\hat{s}} \iota_{\mathbb{L}}\omega = \int_{\gamma} \iota_{\mathbb{L}}\omega.
\]

Next, let $r: [0, 1] \mapsto Z$ be a trajectory of $X_H$, and notice that $X_H$ satisfies $\iota_{X_H}\omega = c \frac{dy}{y} + df$. Let $y\frac{\partial}{\partial y}$ be a representative of $\mathbb{L}$. Because $X_H$ is 1-periodic and homotopic to $\gamma$,
\[
k = \int_r \iota_{\mathbb{L}}\omega  = \int_0^1 \iota_{y\frac{\partial}{\partial y}}\omega(\restr{X_H}{r(t)}) dt = \int_0^1 \restr{-(c\frac{dy}{y} + df) ( y\frac{\partial}{\partial y})}{r(t)} dt = -c.
\]
\end{proof}
\end{proposition}

\subsection{Hamiltonian actions on symplectic and $b$-symplectic manifolds.}

Let $\mathbb{T}^n$ be a torus which acts on a symplectic manifold $M$ by symplectomorphisms, and denote by $\mathfrak{t}$ and $\mathfrak{t}^*$ its Lie algebra and corresponding dual, respectively. We say that the action is \textbf{Hamiltonian} if there exists an invariant map $\mu:M\to\mathfrak{t}^*$ such that for each element $X\in\mathfrak{t}$,
\begin{equation}\label{momentmap}
d\mu^X=\iota_{X^\#}\omega,
\end{equation}
where $\mu^X=<\mu,X>$ is the component of $\mu$ in the direction of $X$, and $X^\#$ is the vector field on $M$ generated by $X$:
\[
X^\#(p)=\frac{d}{dt}\left[\text{exp} (tX)\cdot p\right].
\]
The map $\mu$ is called the \textbf{moment map}.

To study a torus acting on a $b$-manifold by $b$-symplectomorphisms, the definition of a \emph{Hamiltonian} action and of a \emph{moment map} must be adapted. To motivate the appropriate definitions we study two examples.

\begin{example}\label{example:S2}
Consider the $b$-symplectic manifold $(\SSS^2, Z = \{h = 0\}, \omega=\frac{d h}{h}\wedge d\theta)$, where the coordinates on the sphere are  $h\in\left[-1,1\right]$ and $\theta\in\left[0,2\pi\right]$.
For the $\SSS^1$-action given by the flow of $-\frac{\del}{\del \theta}$, a moment map on $M \backslash Z$ is $\mu(h,\theta)= \log |h|$. The image of $\mu$ is drawn in the left half of Figure \ref{fig:S2andT2} as two superimposed half-lines depicted slightly apart to emphasize that each point in the image has two connected components in its preimage: one in the northern hemisphere, and one in the southern hemisphere. By enlarging the codomain of our moment map to include points ``at infinity,'' we can define moment maps for torus actions on a $b$-manifold that enjoy many of the same properties as classic moment maps: they will be everywhere defined and their image will be a parameter space for the orbits of the action.
\begin{figure}[ht]
\includegraphics[height=4.5cm]{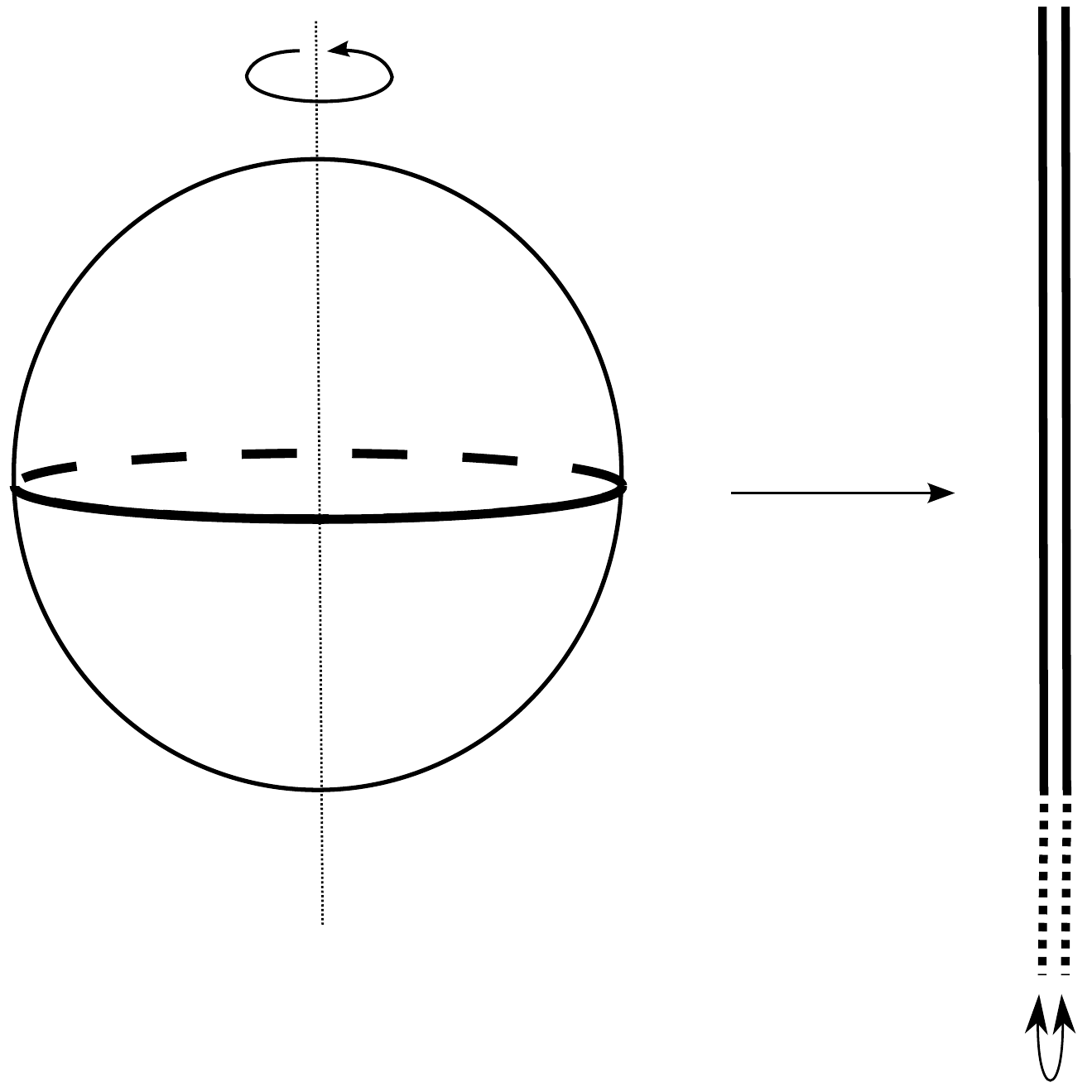} \hskip 1in  \includegraphics[height=4.5cm]{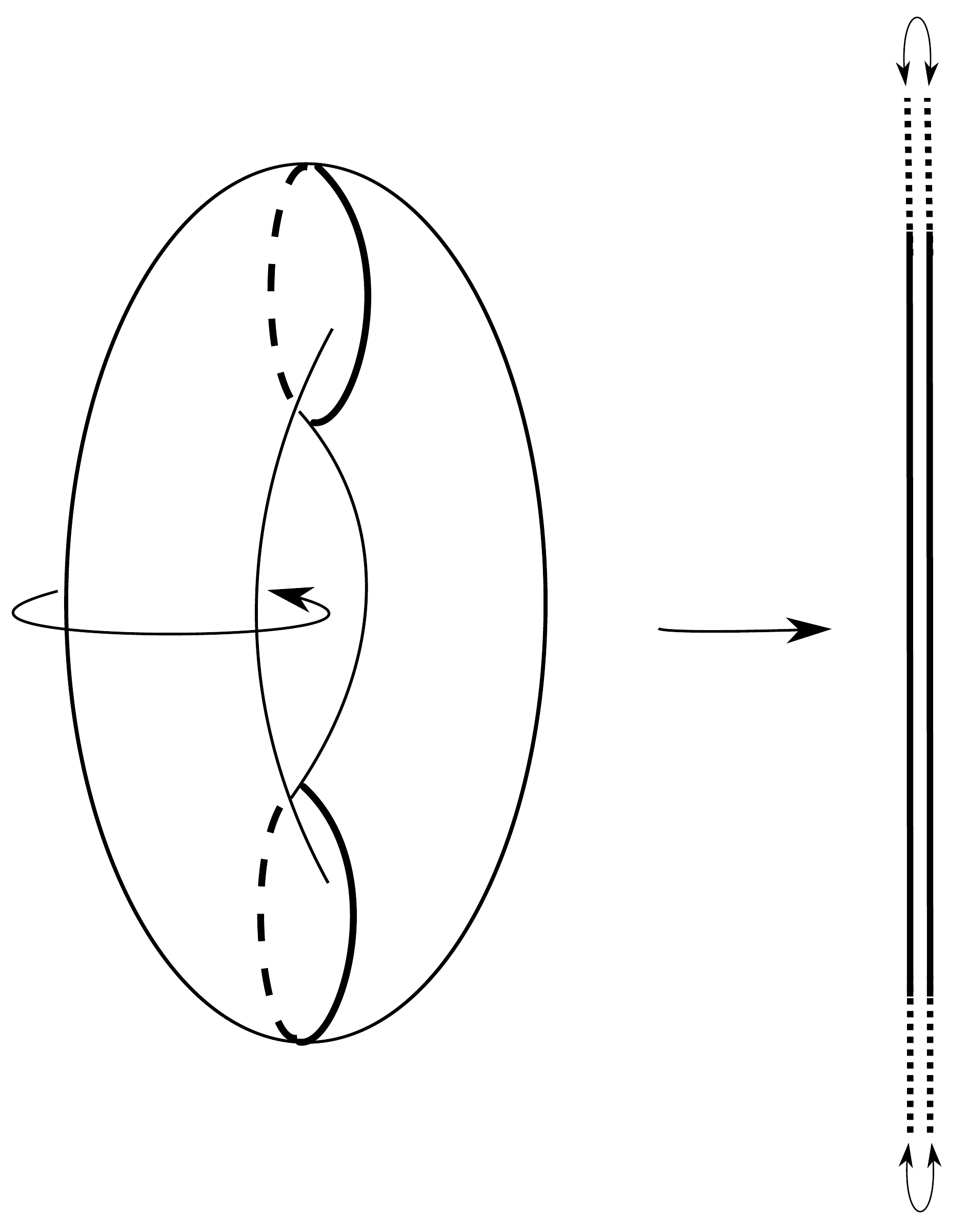}
\caption{The moment maps of the $\mathbb{S}^1$-actions on $M \backslash Z$ in Examples \ref{example:S2} and \ref{example:T2}.}
\label{fig:S2andT2}
\end{figure}
\end{example}

\begin{example}\label{example:T2}
Consider the $b$-symplectic manifold
\[
(\TT^2, Z = \{\theta_1 \in \{0, \pi\}\}, \omega=  \frac{d\theta_1}{\sin\theta_1}\wedge d\theta_2)
\]
where the coordinates on the torus are $\theta_1,\theta_2\in\left[0,2\pi\right]$. The circle action of rotation on the $\theta_2$ coordinate is given by the ${^b}C^{\infty}$ Hamiltonian function $\log \left|\frac{1 + \cos\theta_1}{\sin\theta_1}\right|$, the image of which is shown in the right half of Figure \ref{fig:S2andT2}.
\end{example}

\begin{definition}
An action of $\TT^n$ on a $b$-symplectic manifold $(M,\omega)$ is \textbf{Hamiltonian} if for any $X,Y\in\mathfrak{t}$:
\begin{itemize}
\item the one-form $\iota_{X^\#}\omega$ is exact, i.e., has a primitive $H_X\in\,^bC^\infty(M)$;
\item $\omega(X^\#,Y^\#)=0$.
\end{itemize}
A Hamiltonian action is \textbf{toric} if it is effective and $\dim\left(\mathbb{T}^n\right)=\frac{1}{2}\dim\left(M\right)$.
\end{definition}

\section{ Classification of toric $b$-surfaces}\label{sec:bsurfaces}

In this section we classify toric $\mathbb{S}^1$-actions on $b$-symplectic surfaces, which can be summarized as follows: all $b$-symplectic toric surfaces are equivariantly $b$-symplectomorphic to either Examples \ref{example:S2} or \ref{example:T2} with possibly a different number of components of the exceptional hypersurface. As noted in Remark \ref{rmk:classification in dim 2 and 4}, this result is a consequence of the Delzant theorem for toric $b$-symplectic manifolds (Theorem \ref{thm:bDelzant}). We prove the $2$-dimensional case independently here because it can be considered as the \emph{toy model} that yields the general classification of $b$-symplectic toric manifolds, as explained in Remark \ref{rmk:anotherway}.

One necessary ingredient is the fact that the only orientable compact surfaces admitting an effective $\mathbb{S}^1$-action are $\mathbb{S}^2$ and $\mathbb{T}^2$, and the action is equivalent to the standard action by rotation (see \cite[p.~22]{audin} for a proof). If the surface is symplectic and the action is Hamiltonian, then the surface must be a sphere $\mathbb S^2$. However, there are $b$-symplectic structures on $\mathbb{T}^2$ that admit a Hamiltonian circle action.

Another ingredient is Radko's classification of  $b$-symplectic structures on a compact surface $M$, up to $b$-symplectomorphism, by the set of curves $Z$, their modular periods and the regularized Liouville volume ({see Theorem 3 in \cite{radko} or} Theorem 17 in \cite{guimipi12}). To obtain an equivariant version of this result, we need an equivariant version of the $b$-Moser theorem (the classic version is Theorem 38 in \cite{guimipi12}):

\begin{theorem} \label{equivglobalmoser}\textbf{[Equivariant $b$-Moser theorem]}
Suppose that $M$ is compact and let $\omega_0$ and $\omega_1$ be two $b$-symplectic forms on $(M,Z)$. Suppose that $\omega_t$, for $0\leq t\leq 1$, is a smooth family of $b$-symplectic forms on $(M,Z)$ joining $\omega_0$ and $\omega_1$ such that the $b$-cohomology class $[\omega_t]$ does not depend on $t$. Then, there exists a family of diffeomorphisms $\gamma_t:M\to M$, for $0\leq t\leq 1$ such that $\gamma_t$ leaves $Z$ invariant and $\gamma_t^*\omega_t=\omega_0$.

Furthermore, if the $b$-symplectic forms are invariant by the action of a compact Lie group, then $\gamma_t$ can be made equivariant.
\end{theorem}

\begin{proof}
The proof is essentially the same as in \cite{guimipi12}, but using averaging techniques to guarantee that $\gamma_t$ is equivariant. The isotopy $\gamma_t$ is obtained by integrating a smooth family of invariant $b$-vector fields $v_t$ such that $\cL_{v_t}\omega_t=\frac{d\omega_t}{dt}$. Because  $[\omega_t]$ is independent of $t$, we may pick a family $\mu_t\in\,^b\Omega^1(M)$ such that $\frac{d\omega_t}{dt}=d \mu_t$. By replacing $\mu_t$ with $\int_G (\rho_g^*(\mu_t)dg)$ (where $g$ is a Haar measure and $\rho: G\times (M,Z)\longrightarrow (M,Z)$ is the action), we may assume that $\mu_t$ is invariant. Then the $b$-vector field defined by Moser's equation $\iota_{v_t}\omega_t=-\mu_t$ is invariant because $\mu_t$ and $\omega_t$ are invariant. Its flow gives the desired equivariant diffeomorphism.
\end{proof}

\begin{theorem}
\label{thm:surfaces}
{A $b$-symplectic surface with a toric $\mathbb{S}^1$-action is equivariantly $b$-symplectomorphic to either $(\mathbb{S}^2,Z)$ or $(\mathbb{T}^2,Z)$, where $Z$ is a collection of latitude circles (in the $\mathbb{T}^2$ case, an even number of such circles), the action is the standard rotation, and the $b$-symplectic form is determined by the modular periods of the exceptional curves and the regularized Liouville volume.}
\end{theorem}

\begin{proof}
{
By the classic result about effective circle action on surfaces, we may assume that the surface and group action is either $\mathbb{S}^2$ or $\mathbb{T}^2$ with the standard rotation. Since the Hamiltonian circle action must preserve $Z$, the components of $Z$ must be latitude circles (in the standard coordinates on $\mathbb{S}^2$ or $\mathbb{T}^2$, respectively level curves of $h$ or $\theta_1$). This forces the number of components of $Z \subseteq \mathbb{T}^2$ to be even, since the orientation of a $b$-symplectic form changes when you cross a component of $Z$. By applying an equivariant diffeomorphism, we may also assume that $Z$ is fixed. Finally we repeat the proof of Theorem 17 in \cite{guimipi12}, replacing the classic $b$-Moser theorem (Theorem 38 in \cite{guimipi12}) by its equivariant version (Theorem \ref{equivglobalmoser} above).
}
\end{proof}

\section{Toric actions in higher dimensions}\label{sect:MminusZ}

\subsection{Local picture: in a neighborhood of $Z$}\label{subsect:local}

We begin by studying a toric actions near a connected component of $Z$. To simplify our exposition, \emph{we assume throughout section \ref{subsect:local} that $Z$ has one connected component}. In the general case, these results hold in a neighborhood of each connected component of $Z$.

Proposition \ref{prop:localmomentmap} is the main result of this section. It states that a toric action near $Z$ is locally a product of a codimension-1 torus action on a symplectic leaf of $Z$ with an circle action whose flow is transverse to the leaves. The codimension-1 subtorus $\mathbb{T}^{n-1}_Z$ will consist of those elements of $\mathbb{T}^n$ that preserve the symplectic foliation of $Z$. Toward the goal of showing that this subtorus is well-defined, we recall the following standard fact from Poisson geometry.

\begin{remark}\label{rmk:smoothhamiltonians}
Let $(M, Z, \omega)$ be a $b$-symplectic manifold.  Since $Z$ is a Poisson submanifold of $M$, a Hamiltonian vector field of a smooth function $X_f$ is tangent to the symplectic leaves of $Z$ and $\restr{X_f}{\mathcal{L}} = X_{f \circ i_{\mathcal{L}}}$ where $i_{\mathcal{L}}: (\mathcal{L}, \omega_{\mathcal{L}}) \rightarrow Z$ is the inclusion of a symplectic leaf into $Z$.
\end{remark}

Given a Hamiltonian $\mathbb{T}^k$-action on $(M^{2n}, Z, \omega)$ and any $X \in \mathfrak{t}$, the $b$-form $\iota_{X^{\#}}\omega$ has a ${^b}C^{\infty}$ primitive that can be written in a neighborhood of $Z$ as $c\log|y| + g$, where $y$ is a local defining function for $Z$, $g$ is smooth, and $c \in \mathbb{R}$ depends on $X$. The map $X \mapsto c$ is a well-defined homomorphism and therefore an element $v_{Z}$ of $\mathfrak{t}^* = \textrm{Hom}(\mathfrak{t}, \mathbb{R})$ called the \textbf{modular weight} of $Z$. We will denote by $\mathfrak{t}_{Z}$ the kernel of $v_{Z}$. By Proposition \ref{prop:propertiesofmodularperiod}, the values $\langle v_{Z}, X\rangle$ are integer multiples of the modular period of $Z$ when $X$ is a lattice vector, so the slope of $v_{Z}$ is rational. We will show in Claim \ref{claim:vznonzero} that $v_{Z}$ is nonzero. First, we prove an equivariant Darboux theorem for compact group actions in a neighborhood of a fixed point. Given a fixed point $p$ of an action $\rho:G\times M\longrightarrow M$, we denote by $d\rho$ the linear action defined via the exponential map in a neighborhood of the origin in $T_p M$:  $d\rho(g,v)= d_p(\rho(g))(v)$.

\begin{theorem}\label{thm:equivDTgeneralG}
Let $\rho$ be a $b$-symplectic action of a compact Lie group $G$  on a $b$-symplectic manifold $(M,Z,\omega)$, and let $p \in Z$ be a fixed point of the action.  Then there exist local coordinates $(x_1,y_1,\dots, x_{n-1},y_{n-1}, z,t)$ centered at $p$ such that the action is linear in these coordinates and
\[
\omega =\sum_{i=1}^{n-1} dx_i\wedge dy_i+\frac{1}{z}\,dz\wedge dt.
\]
\end{theorem}

\begin{proof}

{By picking a Riemannian metric we may assume that $\omega$ and $\rho$ live on the $b$-manifold $(T_pM, T_pZ)$.}

By Bochner's theorem \cite{bo}, the action of $\rho$ is locally equivalent to the action of $d\rho$. That is, there are coordinates $(x_1, y_1, \dots, x_{n-1}, y_{n-1}, z, t)$ centered at $\mathbf{0}=(0,0,\ldots,0)$ on which the action is linear. By studying the construction of $\varphi$ in \cite{bo}, we can choose the coordinates so that $T_pZ$ is the coordinate hyperplane $\{z = 0\}$. Also, after a linear change of these coordinates, we may assume that
\[
\restr{\omega}{\mathbf{0}} = \sum_{i=1}^{n-1} dx_i\wedge dy_i+\frac{1}{z}\,dz\wedge dt.
\]

Next, we will perform an equivariant Moser's trick. Let $\omega_0 = \omega$,
$$\omega_1 =\sum_{i=1}^{n-1} dx_i\wedge dy_i+\frac{1}{z}\,dz\wedge dt, \text{ \,\,\,and\,\,\, }\omega_s=s\omega_1+(1-s) \omega_0, \text{ for } s\in [0,1].
$$
Because $\omega_s$ has full rank at $\mathbf{0}$, we may assume (after shrinking the neighborhood) that $\omega_s$ has full rank for all $s$. Let $\alpha$ be a primitive for $\omega_1 - \omega_0$ that vanishes at $\mathbf{0}$ ($\alpha$ is a $b$-form). Then we apply the argument of Theorem \ref{equivglobalmoser} to obtain the desired equivariant $b$-symplectomorphism.
\end{proof}

In the particular case where the group is a torus we obtain the following:

\begin{corollary}\label{thm:equivDT}Consider a fixed point $z \in Z$ of a symplectic $\mathbb{T}^k$-action on $(M, Z, \omega)$. If the isotropy representation on $T_zM$ is trivial, then the action is trivial in a neighborhood of $z$.
\end{corollary}

\begin{claim}\label{claim:vznonzero} Let $(M^{2n}, Z, \omega)$ be a $b$-symplectic manifold with a toric action. Then $v_{Z}$ is nonzero and therefore $\mathfrak{t}_{Z}$ is a hyperplane in $\mathfrak{t}$.
\end{claim}
\begin{proof} Consider a toric action on $(M^{2n}, Z, \omega)$ such that $\iota_{X^{\#}}\omega \in \Omega^1(M)$ for every $X \in \mathfrak{t}$. We prove that this action is not effective.

Let $(\mathcal{L}, \omega_{\mathcal{L}})$ be a leaf of the symplectic foliation of $Z$. By Remark \ref{rmk:smoothhamiltonians} the action on $M$ induces a Hamiltonian torus action on the symplectic manifold $(\mathcal{L}, \omega_{\mathcal{L}})$. Because $\textrm{dim}(\mathcal{L}) = 2n-2$, there must be a subgroup $\SSS^1 \subseteq \mathbb{T}^n$ that acts trivially on $\mathcal{L}$. For any $z \in \mathcal{L}$, the isotropy representation of this $\SSS^1$-action on $T_zM$ restricts to the identity on $T_z\mathcal{L} \subseteq T_zM$ and preserves the subspace $T_zZ$. It therefore induces a linear $\SSS^1$-action on the 1-dimensional vector space $T_zZ / T_z\mathcal{L}$. Any such action is trivial, so the isotropy representation restricts to the identity on $T_zZ$. By the same argument, the isotropy representation on all of $T_zM$ is trivial. By Corollary \ref{thm:equivDT}, the $\SSS^1$-action is the identity on a neighborhood of $z$ and therefore the action is not effective.
\end{proof}

\begin{corollary}\label{cor:notsmooth} If the $b$-symplectic manifold $(M, Z, \omega)$ admits a toric action such that $\iota_{X^{\#}}\omega \in \Omega^1(M)$ for every $X \in \mathfrak{t}$, then $Z = \emptyset$.
\end{corollary}

When $Z$ is not connected, there is a modular weight $v_Z$ for each connected component of $Z$, but we will see in Claim \ref{clm:vflips} that they are nonzero scalar multiples of one another.

\begin{proposition}\label{prop:no_twisting_occurs}
Let $(M^{2n}, Z, \omega)$ be a $b$-symplectic manifold with a toric action. Let $X$ be a representative of a primitive lattice vector of $\mathfrak{t}/\mathfrak{t}_{Z}$ that pairs positively with $v_{Z}$. Then $\langle X, v_{Z}\rangle$ equals the modular period of $Z$.
\end{proposition}
\begin{proof}
By Proposition \ref{prop:propertiesofmodularperiod}, it suffices to prove that a time-1 trajectory of $X^{\#}$ that starts on $Z$, when projected to the $\mathbb{S}^1$ base of the mapping torus $Z$, travels around the loop once. Let $p \in \mathbb{R}^+$ be the smallest number such that $\Phi_p^{X}(\mathcal{L}) = \mathcal{L}$, where $\Phi_p^{X^{\#}}$ is the time-$p$ flow of $X^{\#}$. The condition that $\omega(X^{\#}, Y^{\#}) = 0$ for all $Y \in \mathfrak{t}_{Z}$ implies that the symplectomorphism $\restr{\Phi_p^{X^{\#}}}{\mathcal{L}}$ preserves the $\mathbb{T}_{Z}$-orbits of $\mathcal{L}$. We can realize any such symplectomorphism as the time-1 flow of a Hamiltonian vector field $v$ on the symplectic leaf $(\mathcal{L}, \omega_{\mathcal{L}})$ (see, for example, the proof of Proposition 6.4 in \cite{lertol}). The product of the $\mathbb{T}_{Z}$ action with the flow of $p^{-1}v$ defines a Hamiltonian $\mathbb{T}_{Z} \times \mathbb{S}^1 \cong \mathbb{T}^n$ action on $(\mathcal{L}, \omega_{\mathcal{L}})$, so there exists $\mathbb{S}^1 \subseteq \mathbb{T}_Z \times \mathbb{S}^1$ that acts trivially on $\mathcal{L}$. Since the $\mathbb{T}_{Z}$ action is effective, this $\mathbb{S}^1$ is not a subset of $\mathbb{T}_{Z}$. Therefore we may assume, after replacing $X$ with $X + Y$ for some $Y \in \mathfrak{t}_{Z}$, that the time-$p$ flow of $X^{\#}$ is the identity on $\mathcal{L}$. Then, for any $z \in \mathcal{L}$, the isotropy representation of the time-$p$ flow of $X^{\#}$ would be the identity on $T_zM$, proving (by Corollary \ref{thm:equivDT}) that the time-$p$ flow of $X^{\#}$ is the identity in a neighborhood of $z$. By effectiveness, $p = 1$.
\end{proof}

Proposition \ref{prop:no_twisting_occurs} implies that the trajectories of $X^{\#}$ on $Z$ travel around the $\mathbb{S}^1$ base of the mapping torus $Z$ exactly once. Because $X^{\#}$ is periodic and a Poisson vector field, its flow defines a product structure on $Z$.

\begin{corollary}\label{cor:product} Let $(M^{2n}, Z, \omega)$ be a $b$-symplectic manifold with a toric action and $\mathcal{L}$ a symplectic leaf of $Z$. Then $Z \cong \mathcal{L} \times \mathbb{S}^1.$
\end{corollary}

When $Z$ is not connected, this implies that each component $Z' \subseteq Z$ is of the form $\mathcal{L}'\times \SSS^1$, for possibly distinct $\mathcal{L}'$. We will see that the existence of a global toric action forces all $\mathcal{L}'$ to be identical.

We are nearly ready to prove that in a neighborhood of $Z$ the toric action splits as the product of a $\mathbb{T}_{Z}^{n-1}$-action and an $\SSS^1$-action. We preface this by studying a related example in classic symplectic geometry that will give us intuition for the proofs of Lemma \ref{prop:technicalprop} and Proposition \ref{prop:localmomentmap}.

Consider the symplectic manifold $M = \SSS^2 \times \SSS^2, \omega = dh_1 \wedge d\theta_1 + dh_2 \wedge d\theta_2$ with a Hamiltonian $\mathbb{T}^2$-action defined by
\[
(t_1, t_2) \cdot (h_1, \theta_1, h_2, \theta_2) = (h_1, \theta_1 + t_1, h_2, \theta_2 + t_2).
\]
Let $X_1, X_2\in\mathfrak{t}$ be such that $X_1^{\#} = \frac{\partial}{\partial \theta_1}$ and $X_2^{\#} = \frac{\partial}{\partial \theta_2}$. The moment map of this action is given by $(h_1, h_2)$, its image is the square $\Delta = [-1, 1]^2$.

Consider the hypersurfaces $Z_1 = \{h_2 = 0\}$ and $Z_2 = \{h_1 + h_2 = -1\}$ in $M$, see Figure \ref{fig:hypersurfacesinS2S2}.
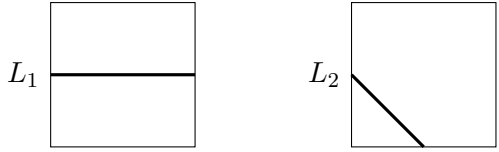
\begin{figure}[ht]
\centering
\begin{tikzpicture} [scale = 0.8]
\pgfmathsetmacro{\boxradius}{1.2}
\pgfmathsetmacro{\xshift}{5}

\draw (-\boxradius, -\boxradius) -- (\boxradius, -\boxradius) -- (\boxradius, \boxradius) -- (-\boxradius, \boxradius) -- cycle;
\draw (\xshift-\boxradius, -\boxradius) -- (\xshift+\boxradius, -\boxradius) -- (\xshift+\boxradius, \boxradius) -- (\xshift-\boxradius, \boxradius) -- cycle;

\draw[very thick] (-\boxradius, 0) node[left] {$L_1$}-- (\boxradius, 0);
\draw[very thick] (\xshift - \boxradius, 0) node[left] {$L_2$} -- (\xshift, -\boxradius);

\end{tikzpicture}
\caption{Hypersurfaces in $\SSS^2 \times \SSS^2$: $Z_1 = \mu^{-1}(L_1)$ and $Z_2 = \mu^{-1}(L_2)$.}
\label{fig:hypersurfacesinS2S2}
\end{figure}
Near $L_1$, the polytope $\Delta$ is locally the product $L_1 \times (-\varepsilon, \varepsilon)$; near $L_2$, it is not locally a product. The vector field $u = \frac{\partial}{\partial h_2}$ in a neighborhood of $Z_1$ has the property that $dh_2(u) = 1$, and $\omega(Y^{\#}, u) = 0$ for all $Y$ in the hyperplane of $\mathfrak{t}$ spanned by $X_1$. If we flow $Z_1$ along the vector field $u$, the image under $\mu$ would consist of the line segment $L_1$ moving with constant velocity in the direction perpendicular to $\langle X_1 \rangle$, showing once again that $\Delta$ is locally the product $L_1 \times (-\varepsilon, \varepsilon)$ near $L_1$. In contrast, there is no vector field $u'$ in a neighborhood of $Z_2$ such that $d(h_1 + h_2)(u') = 1$ and $\omega(Y^{\#}, u') = 0$ for all $Y$ in a hyperplane of $\mathfrak{t}$, reflecting the fact that $\Delta$ is not locally a product near $L_2$. The reason that no such $u'$ exists is because every hyperplane of $\mathfrak{t}$ contains some $Y$ such that $\iota_{Y^{\#}}\omega$ is a multiple of $d(h_1 + h_2)$ somewhere along $Z_2$ (making the condition that $d(h_1 + h_2)(u') = 1$ incompatible with $\omega(Y^{\#}, u') = 0$). In other words, the fact that $\Delta$ is locally a product near $L_1$ is reflected in the fact that there is a hyperplane in $\mathfrak{t}$ such that $\textrm{ker}(\iota_{Y^{\#}}\omega_z) \neq T_zZ$ for all $z \in Z$ and all $Y$ in this hyperplane.

In a neighborhood of the exceptional hypersurface $Z$ of a $b$-manifold, a toric action will always behave similarly to the $Z_1$ example: the hyperplane $\mathfrak{t}_{Z} \subseteq \mathfrak{t}$ satisfies the property $\textrm{ker}(\iota_{Y^{\#}}\omega_z) \neq T_zZ$ for all $z \in Z$ and $Y \in \mathfrak{t}_{Z}$. This fact is the content of Lemma \ref{prop:technicalprop} and will play an important role in the proof of Proposition \ref{prop:localmomentmap}.

\begin{lemma}\label{prop:technicalprop}
Let $k < n$ and consider a Hamiltonian $\mathbb{T}^k$-action on $(M^{2n}, Z, \omega)$ for which $\iota_{X^{\#}}\omega \in \Omega^1(M)$ for each $X \in \mathfrak{t}$. Then for any $z \in Z$ and $X \in \mathfrak{t}$, $\ker(\iota_{X^{\#}}\omega_z) \neq T_zZ$.
\end{lemma}

\begin{proof}
Pick coordinates $(t, \theta, x_1, \dots, x_{2n-2})$ centered around a point $z \in Z$ such that $t$ is a defining function for $Z$, $\frac{\partial}{\partial t}$ is a $\mathbb{T}^k$-invariant vector field, and $i_Z^*(d\theta) = \iota_{\mathbb{L}}(\omega)$. Because $\frac{\partial}{\partial t}$ is invariant, for any $X \in \mathfrak{t}$ we have $0 = \left[\frac{\partial}{\partial t}, X^{\#}\right](\theta) = \frac{\partial}{\partial t} \left( d\theta (X^{\#}) \right)$.
If $\textrm{ker}(\iota_{X^{\#}}\omega_z)$ contains $T_zZ$, then a calculation in local coordinates shows that $\iota_{X^{\#}}\omega$ must in fact vanish completely at $z$. That is, if $\ker(\iota_{X^{\#}}\omega_z) \supseteq T_zZ$, then $\ker(\iota_{X^{\#}}\omega_z) = T_zM$.
\end{proof}

\begin{proposition}\label{prop:localmomentmap}
Let $(M^{2n}, Z, \omega)$ be a $b$-symplectic manifold with a toric action and $\mathcal{L}$ a leaf of its symplectic foliation, and $v_Z$ the modular weight of $Z$. Pick a lattice element $X \in \mathfrak{t}$ that represents a generator of $\mathfrak{t} / \mathfrak{t}_Z$ and pairs positively with $v_{Z}$.

Then there is a neighborhood $\mathcal{L} \times \SSS^1 \times (-\varepsilon, \varepsilon) \cong \mathcal{U} \subseteq M$ of $Z$ such that the $\mathbb{T}^n$-action on $\mathcal{U}\setminus Z$ has moment map
\begin{equation}\label{eqn:localmm}
\mu_{\mathcal{U}\setminus Z}: \mathcal{L} \times \SSS^1 \times \left((-\varepsilon, \varepsilon)\setminus\{0\}\right) \rightarrow \mathfrak{t}^* \cong \mathfrak{t}_{Z}^* \times \mathbb{R},\text{\,\,}
(\ell, \rho, t) \mapsto (\mu_{\mathcal{L}}(\ell), c\log|t|)
\end{equation}
where $c$ is the modular period of $Z$, the map $\mu_{\mathcal{L}}: \mathcal{L} \rightarrow \mathfrak{t}_{Z}^*$ is a moment map for the $\mathbb{T}_{Z}^{n-1}$-action on $\mathcal{L}$, and the isomorphism $\mathfrak{t}^* \cong \mathfrak{t}_{Z}^* \times \mathbb{R}$ is induced by the splitting $\mathfrak{t} \cong \mathfrak{t}_Z \oplus \langle X \rangle $.
\end{proposition}

\begin{proof}
Observe that the splitting $\mathfrak{t} \cong \mathfrak{t}_Z \oplus \langle X \rangle $ induces a splitting $\mathbb{T}^n \cong \mathbb{T}^{n-1} \times \SSS^1$. Pick a primitive $f$ of $\iota_{X^{\#}}\omega$. For a suitable neighborhood $\mathcal{U}$ of $Z$, let $y: \mathcal{U} \rightarrow \mathbb{R}$ be a defining function for $Z$ with the property that $f = c\textrm{log}|y|$ on $\mathcal{U}\setminus Z$. Because $f$ is $\mathbb{T}^n$-invariant, so too is $y$. Our first goal is to pick a vector field $u$ in a neighborhood of $Z$ with the following three properties:
\begin{enumerate}
\item $dy(u) = 1$
\item $\iota_{Y^\#}\omega (u) = 0$ for all $Y \in \mathfrak{t}_Z$
\item $u$ is $\mathbb{T}^n$-invariant
\end{enumerate}

To show that a vector field exists that satisfies conditions $(1)$ and $(2)$ simultaneously, it suffices to observe that for each $z \in Z$ and $Y \in \mathfrak{t}_Z$, $\textrm{ker}(\iota_{Y^{\#}}\omega_z) \neq T_zZ$ by Lemma \ref{prop:technicalprop}. Let $u$ be a vector field satisfying $(1)$ and $(2)$. Because $dy$ and each $\iota_{Y^\#}\omega$ are $\mathbb{T}^n$-invariant, we can average $u$ by the $\mathbb{T}^n$-action   without disturbing properties $(1)$ and $(2)$. By replacing $u$ with its $\mathbb{T}^n$-average, we assume that $u$ is $\mathbb{T}^n$-invariant. Let $\Phi^u_t$ and $\Phi^{X^{\#}}_t$ be the time-t flows of $u$ and $X^{\#}$ respectively. Then, using Corollary \ref{cor:product}, the map
$$\phi: \mathcal{L} \times \SSS^1 \times (-\varepsilon, \varepsilon) \rightarrow \mathcal{U}, \text{\,\,\,\,\,\,}(\ell, \rho, t) \mapsto \Phi^u_t \circ \Phi^{X^{\#}}_\rho (\ell)$$
is a diffeomorphism for sufficiently small $\varepsilon$. Let $p$ and $t$ be the projections of $\mathcal{L} \times \SSS^1 \times (-\varepsilon, \varepsilon)$ onto $Z \cong \mathcal{L} \times \SSS^1$ and $(-\varepsilon, \varepsilon)$ respectively. To study the induced $\mathbb{T}^n$-action on the domain of $\phi$, fix some $(s, g) = (\exp(kX), \exp(Y)) \in \SSS^1 \times \mathbb{T}^{n-1}_Z$ and recall that since $u$ is $\mathbb{T}^n$-invariant, its flows commute with the flows of all $\{X^{\#} \mid X \in \mathfrak{t}\}$.
{We can check that in these coordinates the action is given by}
$$
{\phi(g \cdot_{\mathcal{L}} \ell, \rho + s, t) = \Phi^u_t \circ \Phi^{X^{\#}}_{\rho + s} (g \cdot_{\mathcal{L}} \ell) = (s, g) \cdot \phi(\ell, \rho, t),}
$$
{where $\cdot_{\mathcal{L}}$ denotes the $\mathbb{T}_Z^{n-1}$-action on $\mathcal{L}$.}

We will show that the moment map for this action is given by (\ref{eqn:localmm}). By construction, $\mu_{\mathcal{U}\setminus Z}^{X} \in {^b}C^{\infty}\left(\mathcal{L} \times \SSS^1 \times \left((-\varepsilon, \varepsilon)\setminus\{0\}\right)\right)$ is given by $c\log|t|$ as desired. To prove that $\iota_{Y^{\#}}(\phi^*\omega) = d\mu_{\mathcal{U}\setminus Z}^Y$ for $Y \in \mathfrak{t}_Z$, we define the map
$$p_\mathcal{L}: \mathcal{U}\rightarrow \mathcal{L},\text{\,\,\,\,\,\,}\phi(\ell, \rho, t) \mapsto \ell.$$
Since the map $p_{\mathcal{L}}$ can be realized at $\phi(\ell, \rho, t)$ as the time-($-t$) flow of $u$ followed by the time-($-\rho$) flow of $X^{\#}$, both of which preserve $\iota_{Y^{\#}}\omega$, it follows that $p_{\mathcal{L}}^*(\iota_{Y^{\#}}\omega) = \iota_{Y^{\#}}\omega$. Then {$\iota_{Y^{\#}}(\phi^*\omega) = \phi^* p_{\mathcal{L}}^* (\iota_{Y^{\#}}\omega) = d\mu_{\mathcal{U}\setminus Z}^Y.$}

\end{proof}

As a result, for such a neighborhood $\mathcal{U}$ of $Z$, each half (each connected component) of the the open set $\mathcal{U}\setminus Z$ is taken under the moment map to the product of a Delzant polytope $\Delta$ with an interval of the form $(-\infty,k)$. The image of $Z$ under the moment map, were it to be defined, would be $\Delta\times\{-\infty\}$. In Sections \ref{sec:codomain} and \ref{sec:mmallofM} we will make this precise.

\subsection{Global picture}

Let $(M^{2n}, Z, \omega)$ be a $b$-symplectic manifold with a toric action. We now consider the general case, when $Z$ is not necessarily connected. For a connected component $W$ of $M \setminus Z$, we write $\mu_{W}: W \rightarrow \mathfrak{t}^*$.

\begin{claim}\label{clm:Wconvexity}
The image $\mu_W(W)$ is convex.
\end{claim}

\begin{proof}
Let $Z_1,\ldots,{Z}_{r}$ be the connected components of $Z$ which are in the closure of $W$. By Proposition \ref{prop:localmomentmap}, we can find a function $t_i$ in a neighborhood of $Z_i$ for which an $\SSS^1$ factor of the $\mathbb{T}^n$-action is generated by the Hamiltonian $c_i\log|t_i|$ for $c_i > 0$. Define $W_{\geq \varepsilon}\subseteq W$ to be $W \backslash \{|t_i| < \varepsilon\}$, let $W_{= \varepsilon}$ be its boundary, and let $W_{> \varepsilon} = W_{\geq \varepsilon} \backslash W_{= \varepsilon}$.

Performing a symplectic cut at $W_{= \varepsilon}$ gives a compact symplectic toric manifold $\overline{W_{\geq \varepsilon}}$ which has an open subset canonically identified with $W_{> \varepsilon}$. Let $\mu_{W, \varepsilon}: \overline{W_{\geq \varepsilon}} \rightarrow \mathfrak{t}^*$ be the moment map for the toric action on $\overline{W_{\geq \varepsilon}}$ that agrees with $\mu_W$ on $W_{> \varepsilon}$. To show that $\mu_W(W)$ is convex, pick points $\mu_W(p), \mu_W(q)$ in $\mu_W(W)$ and fix some $\varepsilon > 0$ small enough that that $p, q \in W_{> \varepsilon}$. Because $\overline{W_{\geq \varepsilon}}$ is compact, $\mu_{W, \varepsilon}(\overline{W_{\geq \varepsilon}})$ contains the straight line joining $\mu_W(p) = \mu_{W, \varepsilon}(p)$ and $\mu_W(q) = \mu_{W, \varepsilon}(q)$. Since $\mu_{W, \varepsilon}(\overline{W_{\geq \varepsilon}}) \subseteq \mu_{W}(W)$, the image $\mu_W(W)$ also contains the straight line joining $\mu_W(p)$ and $\mu_W(q)$.
\end{proof}

As a result of Proposition \ref{prop:localmomentmap}, for each connected component $Z'$ of $Z$ adjacent to $W$ there is a neighborhood $\mathcal{U}$ of $Z'$ such that $\mu_W(\mathcal{U} \cap W)$ is the product of a Delzant polytope with the ray generated by $-v_{Z'}$. By performing symplectic cuts near the hypersurfaces adjacent to $W$ (as in the proof of Claim \ref{clm:Wconvexity}) to partition the image of $\mu_W$ into a convex set and these infinite prisms, we see that the convex set $\mu_W(W)$ extends indefinitely in precisely the directions $-v_{Z'}$ for all components $Z'$ adjacent to $W$.

By convexity, if $\mu_W(W)$ extends infinitely far in directions $v_1$ and $v_2$, then it also extends infinitely far in every direction of the cone spanned by $v_1$ and $v_2$. Since the number of these directions is bounded by the (finite) number of components of $Z$ adjacent to $W$, $v_1$ and $v_2$ must be multiples of one another. This proves that all $v_{Z}$ occupy the same one-dimensional subspace of $\mathfrak{t}^*$, so that $\mathfrak{t}_{Z'} := v_{Z'}^{\perp}$ is independent of the choice of component $Z' \subseteq Z$.

\begin{claim}\label{clm:vflips}
Suppose that $Z_1$ and $Z_2$ are two different connected components of $Z$ both adjacent to the same connected component $W$ of $M \backslash Z$. Then $v_{Z_1} = kv_{Z_2}$ for some $k < 0$.
\end{claim}

\begin{proof}
By the discussion above, $v_{Z_1} = kv_{Z_2}$ for some $k \in \mathbb{R}$, and by Claim \ref{claim:vznonzero}, $k \neq 0$. It suffices, therefore, to prove that $k$ cannot be positive. Assume towards a contradiction that $k$ is positive, and pick $X \in \mathfrak{t}$ such that $\langle X, v_{Z_1}\rangle > 0$, and let $H: W \rightarrow \mathbb{R}$ be a Hamiltonian for the flow of $X^{\#}$. By performing symplectic cuts sufficiently close to the components of $Z$ adjacent to $W$ (as in the proof of Claim \ref{clm:Wconvexity}) and using the fact that the level sets of moment maps on compact connected symplectic manifolds are connected, it follows that $H^{-1}(\lambda)$ is connected for any $\lambda \in \mathbb{R}$. In a neighborhood of $Z_1$ and of $Z_2$, the function $H$ approaches negative infinity. Therefore, for sufficiently large values of $N$, the level set $H^{-1}(-N)$ has a connected component inside a neighborhood of $Z_1$ and another inside a neighborhood of $Z_2$. Because $H^{-1}(-N)$ has just one connected component, $Z_1 = Z_2$.
\end{proof}

In particular, each component of $M \backslash Z$ is adjacent to at most two connected components of $Z$.

\begin{definition} The \textbf{weighted adjacency graph} $\mathcal{G} = (G, w)$ of a symplectic $b$-manifold $(M, Z, \omega)$ consists of a graph $G=(V,E)$ and a weight function on the set of edges, $w:E \rightarrow \mathfrak{t}^*$. The graph $G$ has a vertex for each component of $M \setminus Z$ and an edge for each connected component of $Z$ that connects the vertices corresponding to the components of $M \setminus Z$ that it separates. The weight $w(e)$ is the modular weight of the connected component of $Z$ corresponding to $e$.
\end{definition}

When $(M^{2n}, Z, \omega)$ has an effective toric action, this graph must either a loop or a line, as illustrated in Figure \ref{fig:graphsofWs}. If it is a loop, Claim \ref{clm:vflips} implies that it must have an even number of vertices.

\begin{figure}[ht]
\centering
\begin{tikzpicture}[scale = 0.8]

\pgfmathsetmacro{\sizer}{1.7}

\DrawTwoDonut{1.45*\sizer}{-0.085*\sizer} {0.26*\sizer}{0.54*\sizer}{0}{red}{very thick}
\DrawTwoDonut{0.55*\sizer}{0.535*\sizer}  {0.21*\sizer}{0.42*\sizer}{70}{red}{very thick}
\DrawTwoDonut{-0.55*\sizer}{0.535*\sizer} {0.21*\sizer}{0.42*\sizer}{110}{red}{very thick}
\DrawTwoDonut{-1.45*\sizer}{-0.085*\sizer}{0.26*\sizer}{0.54*\sizer}{180}{red}{very thick}
\DrawTwoDonut{-0.55*\sizer}{-0.6*\sizer}  {0.18*\sizer}{0.36*\sizer}{250}{red}{very thick}
\DrawTwoDonut{0.55*\sizer}{-0.6*\sizer}   {0.18*\sizer}{0.36*\sizer}{290}{red}{very thick}
\DrawDonut{0}{0}{1*\sizer}{2*\sizer}{0}{black}{very thick}

\draw (1.3*\sizer, 0.4*\sizer) node[]    {$W_0$};
\draw (0*\sizer, 0.7*\sizer) node[]      {$W_1$};
\draw (-1.3*\sizer, 0.4*\sizer) node[]   {$W_2$};
\draw (-1.15*\sizer, -0.5*\sizer) node[] {$W_3$};
\draw (0*\sizer, -0.65*\sizer) node[]    {$W_4$};
\draw (1.15*\sizer, -0.5*\sizer) node[]  {$W_5$};

\pgfmathsetmacro{\rightshift}{7.5}

\draw (\rightshift + 1.5,  0.4) node[above] (A1) {$W_0$};
\draw (\rightshift + 0  ,  0.8) node[above] (A2) {$W_1$};
\draw (\rightshift - 1.5,  0.4) node[above] (A3) {$W_2$};
\draw (\rightshift - 1.5, -0.4) node[below] (A4) {$W_3$};
\draw (\rightshift + 0  , -0.8) node[below] (A5) {$W_4$};
\draw (\rightshift + 1.5, -0.4) node[below] (A6) {$W_5$};
\draw[very thick] (A1) -- (A2) -- (A3) -- (A4) -- (A5) -- (A6) -- (A1);

\pgfmathsetmacro{\downshift}{-3}

\DrawTwoDonut{-1.4}{\downshift}{0.4}{1.1}{90}{red}{very thick}
\DrawTwoDonut{0}   {\downshift}{0.4}{1.1}{90}{red}{very thick}
\DrawTwoDonut{1.4} {\downshift}{0.4}{1.1}{90}{red}{very thick}

\draw[very thick] (-2, 1.1+\downshift) -- (2, 1.1+\downshift) arc (90:-90:1.1 and  1.1) -- (-2, -1.1+\downshift) arc (270:90:1.1 and 1.1);

\draw (-2.2, \downshift) node[] {$W_0$};
\draw (-0.7, \downshift) node[] {$W_1$};
\draw (0.7, \downshift) node[] {$W_2$};
\draw (2.2, \downshift) node[] {$W_3$};

\draw (\rightshift - 2.25,  \downshift) node[] (B1) {{\small{$W_0$}}};
\draw (\rightshift - 0.75,  \downshift) node[] (B2) {{\small{$W_1$}}};
\draw (\rightshift + 0.75,  \downshift) node[] (B3) {{\small{$W_2$}}};
\draw (\rightshift + 2.25,  \downshift) node[] (B4) {{\small{$W_3$}}};
\draw[very thick] (B1) -- (B2) -- (B3) -- (B4);

\pgfmathsetmacro{\massyzero}{\downshift - 1.1}
\pgfmathsetmacro{\massxzero}{2}
\pgfmathsetmacro{\massyscale}{0.02}
\pgfmathsetmacro{\massxscale}{0.02}

\end{tikzpicture}
\caption{The adjacency graph is either a cycle of even length or a line.}
\label{fig:graphsofWs}
\end{figure}
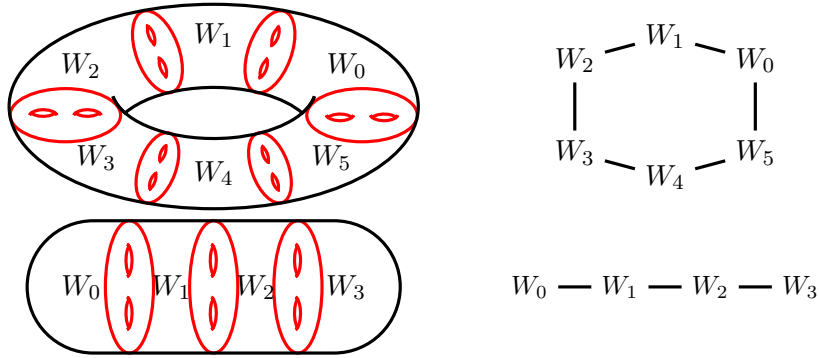

\section{The $b$ moment codomain}\label{sec:codomain}

When $b$-functions are the Hamiltonians of a torus action, we cannot expect to be able to gather them into a moment map $\mu: M \rightarrow \mathfrak{t}^*$ the same way we do in classic symplectic geometry: it is impossible to define $\mu$ along $Z$. In this section, we define a moment map for a toric action on a $b$-manifold. We start with a copy of $\mathfrak{t}^*$ for each component of $M \backslash Z$, add points ``at infinity,'' then glue these copies of $\mathfrak{t}^*$ together in a zig-zag pattern and use labels on the points at infinity to put a smooth structure on the result. Figure \ref{fig:bline} shows an example of the $\mathfrak{t}^* \cong \mathbb{R}$ case.
\begin{figure}[ht]
\centering
\begin{tikzpicture}[scale = 0.8]
\foreach \copies in {0, 1, 2}{ \pgfmathtruncatemacro{\lowerlabel}{(\copies * 2)+1}	 \pgfmathtruncatemacro{\upperlabel}{(\copies * 2) + 2}	
\foreach \paramt in {-3, -2.75, ..., 0.01}	
	{
	\pgfmathsetmacro{\height}{1.5 * exp(\paramt)}	
	\pgfmathsetmacro{\basept}{\copies * (8/3)}	
	\pgfmathsetmacro{\xshift}{0.25*(\height * (\height - 3)) - 0.1}
	\draw[fill = black] (\basept + \xshift, \height - 1.5) circle(.1mm);
	\draw[fill = black] (\basept - \xshift, \height - 1.5) circle(.1mm);
	\draw[red, fill = red] (\basept, -1.6) circle(.5mm) node[below right, black] {\small{$\textrm{w}(e_\lowerlabel)$}};
	
	\draw[fill = black] (\basept + 1.33 + \xshift, 1.5 - \height) circle(.1mm);
	\draw[fill = black] (\basept + 1.33 - \xshift, 1.5 - \height) circle(.1mm);
	\draw[red, fill = red] (\basept + 1.33, 1.6) circle(.5mm) node[above right, black] {\small{$\textrm{w}(e_\upperlabel)$}};
	}}

\foreach \paramt in {-3, -2.75, ..., 0.01}	
	{
	\pgfmathsetmacro{\height}{1.5 * exp(\paramt)}	
	\pgfmathsetmacro{\basept}{0 * (8/3)}	
	
	\draw[fill = black] (\basept  - 0.66,1.5 -  \height) circle(.1mm);

	\pgfmathsetmacro{\basept}{2.5 * (8/3)}	
	\draw[fill = black] (\basept + 0.66, \height -  1.5) circle(.1mm);

	}

\pgfmathsetmacro{\rline}{11}	

\draw (9, 0) edge node[above] {$\hat{x}$} (10.5, 0);
\draw[->] (9, 0) -- (10.5, 0);

\foreach \paramt in {-3, -2.75, ..., 0.01}	
	{
	\pgfmathsetmacro{\height}{1.5 * exp(\paramt)}	
	\pgfmathsetmacro{\xshift}{0.25*(\height * (\height - 3)) - 0.1}
	\draw[fill = black] (\rline, \height - 1.5) circle(.1mm);
	\draw[fill = black] (\rline, 1.5 - \height) circle(.1mm);
	}

\draw (\rline, 1.5) node[right] {$\mathfrak{t}^* \cong \mathbb{R}$};
	
\end{tikzpicture}
\caption{A $b$ moment codomain when $G$ is a line with six edges.}
\label{fig:bline}
\end{figure}
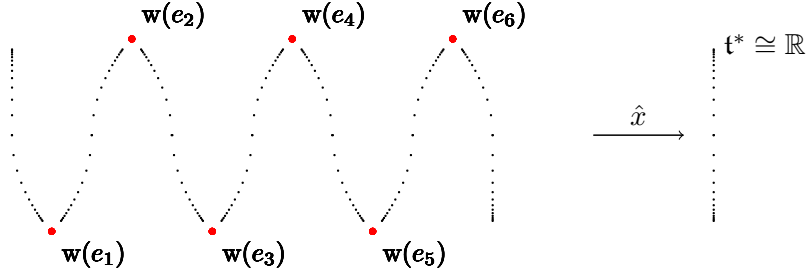

The technical details are as follows: let $\mathbb{T}^n$ be a torus and consider the pair $\mathcal{G} = (G, w)$, of $G$ a finite graph that is either a cycle of even length or a line and a nonvanishing function $w: E \rightarrow \mathfrak{t}^*$ such that whenever the edges $e$ and $e'$ meet at a vertex, $w(e) = kw(e')$ for some $k < 0$. Note that these are exactly the properties that weighted adjacency graphs of toric $b$-symplectic manifolds satisfy.

From this data, we construct the \textbf{$b$ moment codomain} $(\mathcal{R}_\mathcal{G}, \mathcal{Z}_\mathcal{G}, \hat{x})$, where $(\mathcal{R}_\mathcal{G}, \mathcal{Z}_\mathcal{G})$ will be a $b$-manifold and $\hat{x}: \mathcal{R}_\mathcal{G} \backslash \mathcal{Z}_\mathcal{G} \rightarrow \mathfrak{t}^*$ a smooth map. If $G$ is a single vertex, define $(\mathcal{R}_\mathcal{G}, \mathcal{Z}_\mathcal{G}, \hat{x}) = (\mathfrak{t}^*, \emptyset, \textrm{id})$. If $G$ has more than one vertex, let $\mathfrak{t}_w = (w(e))^\perp \subseteq \mathfrak{t}$ for any choice of $e \in E$. As a set, define
\begin{align*}
\mathcal{R}_\mathcal{G} &= \mathfrak{t}^* \times V \sqcup \mathfrak{t}_w^* \times E\\
\mathcal{Z}_\mathcal{G} &= \hspace{1.5cm} \mathfrak{t}_w^* \times E
\end{align*}
and let $\hat{x}((x, v)) = x$.

We define a smooth structure on $\mathcal{R}_\mathcal{G}$ by declaring the maps $\hat{x}$ and the following maps $\{y_{A,e}\}_{e \in E}$ to be smooth: pick an $A \in \mathfrak{t} \backslash \mathfrak{t}_w$ (the smooth structure on $\mathcal{R}_\mathcal{G}$ can be proven to be independent of this choice); for each edge $e=(v,v')$ we define the map $y_{A, e}$ from a subset of $\mathcal{R}_\mathcal{G}$ to $\mathfrak{t}_w^* \times \mathbb{R}$ by

\begin{figure}[ht]
\centering
\begin{tikzpicture}
\draw (5, 3) [->]   node[left] {$\mathfrak{t}^* \times \{v\} \ \sqcup \ \mathfrak{t}^* \times \{v'\} \ \sqcup \ \mathfrak{t}_w^* \times \{e\}$} to (7, 3)   node[right] {$\mathfrak{t}_w^* \times \mathbb{R}$};

\draw [|->] (0.3, 2) node[left] {$(x, v)$} to (7, 2) node[right] {$\left([x], \textrm{exp}\left( \frac{\langle x, A\rangle}{\langle  w(e), A\rangle}\right)\right)$};

\draw [|->] (2.5, 1) node[left] {$(x, v')$} to  (7, 1) node[right] {$\left([x], -\textrm{exp}\left(\frac{\langle x, A\rangle}{\langle  w(e), A\rangle}\right)\right)$};

\draw [|->] (4.6, 0) node[left] {$([x], e)$} to (7, 0) node[right] {$([x], 0),$};
\end{tikzpicture}
\end{figure}
\noindent where $[x]$ is the image of $x$ in $\mathfrak{t}_w^* \cong \mathfrak{t}^* / \mathfrak{t}_w^{\perp}$. This completes the construction of the $b$ moment codomain.\\

Given a smooth $b$-map $\mu$ from a $b$-manifold $(M, Z)$ to a $b$ moment codomain, and a choice of $X \in \mathfrak{t}$, the map $\langle \hat{x} \circ \mu, X \rangle$ defines a $b$-function on $M$. In this way, we can represent many $b$-functions as maps to $b$ moment codomains.

\begin{definition}\label{def:mm}
Consider a Hamiltonian $\mathbb{T}^n$-action on a $b$-symplectic manifold $(M, Z, \omega)$, and let $\mu: M \rightarrow \mathcal{R}_\mathcal{G}$ be a smooth $\mathbb{T}^n$-invariant $b$-map. We say that $\mu$ is a \textbf{moment map} for the action if the map $\mathfrak{t}\ni X \mapsto \mu^X\in C^\infty(M)$, with $\mu^X(p) = \langle \hat{x} \circ \mu(p), X\rangle$, is linear and
$$\iota_{X^{\#}}\omega = d\mu^X.$$
\end{definition}

\begin{example}\label{ex:mms2}
Let $(h, \theta)$ be the standard coordinates on $\SSS^2$. Pick a generator $X$ of $\mathfrak{t} = \textrm{Lie}(\mathbb{S}^1)$ and consider the $\mathbb{S}^1$ action described by $X^{\#} = -\frac{\partial}{\partial \theta}$. For any $c \in \mathbb{R}_{> 0}$, the form $\omega_c = c\frac{dh}{h} \wedge d\theta$ is $b$-symplectic, and for any $k \in \mathbb{R}$ the $b$-function $c\log|h| + k$ generates the $\SSS^1$-action given by the flow of $-\frac{\partial}{\partial \theta}$. Let $G$ consist of a single edge connecting two vertices, and $w(e) = cX^*$. Figure \ref{fig:moment_maps_on_s2} shows the map $\mu: \SSS^2 \rightarrow \mathcal{R}_\mathcal{G}$ corresponding to the Hamiltonian $c\log|h|$, and another $\mu'$ corresponding to $c\log|h| - 2$. In both cases, we have drawn $\mathcal{R}_{\mathcal{G}}$ twice -- the first is vertically so that $\mu$ can be visualized as a projection, the second is bent to look visually similar to Figure \ref{fig:bline}.

\begin{figure}[ht]
\centering
\begin{tikzpicture}[scale = 0.8]

\def\R{1.6}
\def\angEl{10}

\pgfmathsetmacro{\circlebbase}{8}
\pgfmathsetmacro{\momentmapabase}{2.5}
\pgfmathsetmacro{\momentmapbbase}{4}

\pgfmathsetmacro{\momentmapcbase}{2.5 + \circlebbase}
\pgfmathsetmacro{\momentmapdbase}{4 + \circlebbase}

\draw (0, 0) circle (\R);
\DrawLatitudeCircle[\R]{0}{0}

\draw (\circlebbase, 0) circle (\R);
\DrawLatitudeCircle[\R]{0}{5}

\foreach \paramt in {-3, -2.75, ..., 0.01}	
	{
	\pgfmathsetmacro{\height}{1.5 * exp(\paramt)}
	\draw[fill = black] (\momentmapabase, \height - 0 + 0.1) circle(.1mm);
	\draw[fill = black] (\momentmapabase, 0 - \height - 0.1) circle(.1mm);
	
	\pgfmathsetmacro{\xshift}{0.25*(\height * (\height - 3)) - 0.1}
	\draw[fill = black] (\momentmapbbase + \xshift, \height - 1.5) circle(.1mm);
	\draw[fill = black] (\momentmapbbase - \xshift, \height - 1.5) circle(.1mm);
	
	\draw[fill = black] (\momentmapcbase, \height - 0 + 0.1) circle(.1mm);
	\draw[fill = black] (\momentmapcbase, 0 - \height - 0.1) circle(.1mm);
	
	\draw[fill = black] (\momentmapdbase + \xshift, \height - 1.5) circle(.1mm);
	\draw[fill = black] (\momentmapdbase - \xshift, \height - 1.5) circle(.1mm);
	}

\foreach \extratail in {0.25, 0.5}
	{
	\pgfmathsetmacro{\extrat}{\extratail * 1.5}
	\draw[fill = black] (\momentmapabase, 0.1 + 1.5 + \extrat) circle(.1mm);
	\draw[fill = black] (\momentmapabase, -0.1 - 1.5 - \extrat) circle(.1mm);
	
	\draw[fill = black] (\momentmapbbase - .66, 0 + \extrat) circle(.1mm);
	\draw[fill = black] (\momentmapbbase + .66, 0 + \extrat) circle(.1mm);
	
	\draw[fill = black] (\momentmapcbase, 0.1 + 1.5 + \extrat) circle(.1mm);
	\draw[fill = black] (\momentmapcbase, -0.1 - 1.5 - \extrat) circle(.1mm);
	
	\draw[fill = black] (\momentmapdbase - .66, 0 + \extrat) circle(.1mm);
	\draw[fill = black] (\momentmapdbase + .66, 0 + \extrat) circle(.1mm);
	}

\draw[red, fill = red] (\momentmapbbase, -1.6) circle(.3mm) node[below right, text=black] {\small{$c$}};
\draw[red, fill = red] (\momentmapabase, 0) circle(.3mm) node[right, text=black] {\small{$c$}};
\draw[red, fill = red] (\momentmapdbase, -1.6) circle(.3mm) node[below right, text=black] {\small{$c$}};
\draw[red, fill = red] (\momentmapcbase, 0) circle(.3mm) node[right, text=black] {\small{$c$}};
				
\draw[line width = 6pt, join = round, opacity=0.2] (\momentmapabase, 1.6) -- (\momentmapabase, -1.6);
\draw[line width = 6pt, join = round, opacity=0.2] (\momentmapbbase - .66, 0) -- (\momentmapbbase - .63,  1.16 - 1.5) -- (\momentmapbbase - .575,  .91 - 1.5) -- (\momentmapbbase - .5, .71-1.5) -- (\momentmapbbase - .44, .55-1.5) -- (\momentmapbbase, -1.65) -- (\momentmapbbase + .44, .55-1.5) -- (\momentmapbbase + .5, .71-1.5) -- (\momentmapbbase + .575,  .91 - 1.5) -- (\momentmapbbase + .63,  1.16 - 1.5) -- (\momentmapbbase + .66, 0);

\draw[line width = 6pt, join = round, opacity=0.2] (\momentmapcbase, -1.01) -- (\momentmapcbase, 1.01);

\draw[line width = 6pt, join = round, opacity=0.2] (\momentmapdbase - .575,  .91 - 1.5) -- (\momentmapdbase - .5, .71-1.5) -- (\momentmapdbase - .44, .55-1.5) -- (\momentmapdbase, -1.65) -- (\momentmapdbase + .44, .55-1.5) -- (\momentmapdbase + .5, .71-1.5) -- (\momentmapdbase + .575,  .91 - 1.5);

\draw[thick, ->] (1.3, 1.4) -- (2.2, 1.4);
\draw[thick, ->] (1.3, -1.4) -- (2.2, -1.4);

\draw[thick, ->] (1.3 + \circlebbase, 1.4) -- (2.2 + \circlebbase, 0.9);
\draw[thick, ->] (1.3 + \circlebbase, -1.4) -- (2.2 + \circlebbase, -0.9);
		
\draw (2, -3) node {\small{$c\log|h| + 0$ as a map to $\mathcal{R}_{\mathcal{G}}$}};	
	
\draw (2 + \circlebbase, -3) node {\small{$c\log|h| - 2$ as a map to $\mathcal{R}_{\mathcal{G}}$}};	

\end{tikzpicture}
\caption{Two Hamiltonians generating the same $\SSS^1$-action.}
\label{fig:moment_maps_on_s2}
\end{figure}
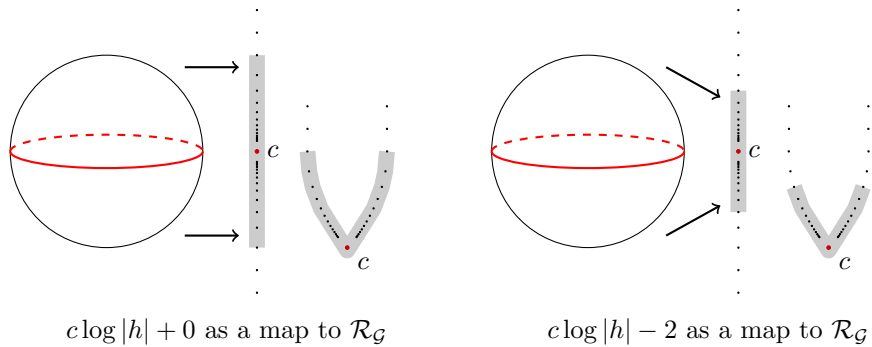

Observe that the moment maps $\mu = c\log|h|$ for different values of $c$ have visually similar images -- the weight data on the graph and therefore the smooth structure of $\mathcal{R}_{\mathcal{G}}$ distinguishes them. For different values of $c$, the $b$-manifolds $(M, Z, \omega_c)$ are \emph{not} symplectomorphic. Also observe that the moment maps in Figure \ref{fig:moment_maps_on_s2} differ by changing the corresponding ${^b}C^{\infty}$ function by a constant; the image of the two moment maps are translates of one another.
\end{example}

\begin{example}
Consider the $b$-symplectic manifold
\[
(M = \SSS^2 \times \SSS^2, Z = \{h_1 = 0\}, \omega = 3\frac{dh_1}{h_1}\wedge d\theta_1 + dh_2 \wedge d\theta_2)
\]
where $(h_1, \theta_1, h_2, \theta_2)$ are the standard coordinates on $\SSS^2 \times \SSS^2$. The $\mathbb{T}^2$-action
\[
(t_1, t_2)\cdot (h_1, \theta_1, h_2, \theta_2) = (h_1, \theta_1- t_1, h_2, \theta_2 - t_2)
\]
is Hamiltonian. Let $X_1$ and $X_2$ be the elements of $\mathfrak{t}$ satisfying $X_1^{\#} = -\frac{\partial}{\partial \theta_1}$ and $X_2^{\#} = -\frac{\partial}{\partial \theta_2}$ respectively. Let $G$ be the connected graph with one edge $e$, and let $w(e) = 3 X_1^*$. Then the map (described here using the basis $\{X_1, X_2\}$)
\[
M \backslash Z \rightarrow \mathfrak{t}^*, \text{\,\,\,\,\,\,} (h_1, \theta_1, h_2, \theta_2) \mapsto (\log|h_1|, h_2),
\]
extends to a moment map $\mu: M \rightarrow \mathcal{R}_{\mathcal{G}}$, the image of which is illustrated in Figure \ref{fig:momentmapS2S2}.

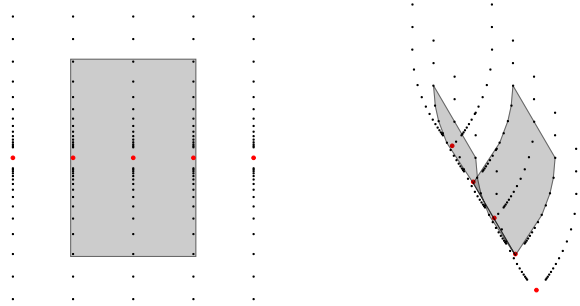
\begin{figure}[ht]
\centering
\begin{tikzpicture}[scale = 0.8]

\pgfmathsetmacro{\startx}{0} \pgfmathsetmacro{\startxt}{6} \pgfmathsetmacro{\verline}{0} \pgfmathsetmacro{\verlinetwo}{-1}
\draw[black, fill = black, opacity = 0.2] (-1.04, 1.64) -- (1.04, 1.64) -- (1.04, -1.64) -- (-1.04, -1.64) -- cycle;
\draw[black, opacity = 0.5] (-1.04, 1.64) -- (1.04, 1.64) -- (1.04, -1.64) -- (-1.04, -1.64) -- cycle;

\foreach \horizshift in {-2, -1, 0, 1, 2}{
	\pgfmathsetmacro{\xshifttwo}{0.35*\horizshift}
	\pgfmathsetmacro{\yshifttwo}{-0.6*\horizshift}
\foreach \paramt in {-3, -2.75, ..., 0.01}	
	{
	\pgfmathsetmacro{\height}{1.5 * exp(\paramt)}
	\draw[fill = black] (\startx + \horizshift, \verline + \height + 0.1) circle(.1mm);
	\draw[fill = black] (\startx + \horizshift, \verline - \height - 0.1) circle(.1mm);
	
	\pgfmathsetmacro{\xshift}{0.25*(\height * (\height - 3)) - 0.1}
	
	\draw[fill = black] (\startxt + \xshifttwo + \xshift, \verlinetwo + \yshifttwo + \height + 0.1) circle(.1mm);
	\draw[fill = black] (\startxt + \xshifttwo - \xshift, \verlinetwo + \yshifttwo + \height + 0.1) circle(.1mm);
	}
	\foreach \extratail in {0.25, 0.5}
		{
		\pgfmathsetmacro{\extrat}{\extratail * 1.5}
		\draw[fill = black] (\startx + \horizshift, \verline + 0.1 + 1.5 + \extrat) circle(.1mm);
		\draw[fill = black] (\startx + \horizshift, \verline -0.1 - 1.5 - \extrat) circle(.1mm);
		
		\draw[fill = black] (\startxt + \xshifttwo - .66 , \verlinetwo + \yshifttwo + 1.6 + \extrat) circle(.1mm);
		\draw[fill = black] (\startxt + \xshifttwo + .66, \verlinetwo + \yshifttwo + 1.6 + \extrat) circle(.1mm);
		
		}
	
	\draw[red, fill = red] (\startx + \horizshift, \verline) circle(.3mm);
	\draw[red, fill = red] (\startxt + \xshifttwo, \verlinetwo + \yshifttwo) circle(.3mm);
	
	}
	
	\pgfmathsetmacro{\xshiftera}{0.35}
	\pgfmathsetmacro{\yshiftera}{-0.6 + 0.1}

	\draw[black, join = round, opacity=0.5]				 (\startxt + \xshiftera, \verlinetwo + \yshiftera -0.12) -- (\startxt + \xshiftera + .44, \verlinetwo + \yshiftera + .55) -- (\startxt + \xshiftera + .5, \verlinetwo + \yshiftera + .71)  -- (\startxt + \xshiftera + .575,  \verlinetwo + \yshiftera + .91) -- (\startxt + \xshiftera + .63,  \verlinetwo + \yshiftera + 1.16) -- (\startxt + \xshiftera + .66,  \verlinetwo + \yshiftera + 1.5) -- (\startxt - \xshiftera + .66,  \verlinetwo - \yshiftera + 0.2 + 1.5) -- (\startxt - \xshiftera + .63,  \verlinetwo - \yshiftera + 0.2 + 1.16) -- (\startxt - \xshiftera + .575,  \verlinetwo - \yshiftera + 0.2 + .91) -- (\startxt - \xshiftera + .5, \verlinetwo - \yshiftera + 0.2 + .71) -- (\startxt - \xshiftera + .44, \verlinetwo - \yshiftera + 0.2 + .55) -- (\startxt - \xshiftera, \verlinetwo - \yshiftera + 0.2 -0.12) -- cycle;
	
	\draw[black, join = round, opacity=0.2, fill = black]				 (\startxt + \xshiftera, \verlinetwo + \yshiftera -0.12) -- (\startxt + \xshiftera + .44, \verlinetwo + \yshiftera + .55) -- (\startxt + \xshiftera + .5, \verlinetwo + \yshiftera + .71)  -- (\startxt + \xshiftera + .575,  \verlinetwo + \yshiftera + .91) -- (\startxt + \xshiftera + .63,  \verlinetwo + \yshiftera + 1.16) -- (\startxt + \xshiftera + .66,  \verlinetwo + \yshiftera + 1.5) -- (\startxt - \xshiftera + .66,  \verlinetwo - \yshiftera + 0.2 + 1.5) -- (\startxt - \xshiftera + .63,  \verlinetwo - \yshiftera + 0.2 + 1.16) -- (\startxt - \xshiftera + .575,  \verlinetwo - \yshiftera + 0.2 + .91) -- (\startxt - \xshiftera + .5, \verlinetwo - \yshiftera + 0.2 + .71) -- (\startxt - \xshiftera + .44, \verlinetwo - \yshiftera + 0.2 + .55) -- (\startxt - \xshiftera, \verlinetwo - \yshiftera + 0.2 -0.12) -- cycle;

\draw[black, join = round, opacity=0.5]				 (\startxt + \xshiftera, \verlinetwo + \yshiftera -0.12) -- (\startxt + \xshiftera - .44, \verlinetwo + \yshiftera + .55) -- (\startxt + \xshiftera - .5, \verlinetwo + \yshiftera + .71)  -- (\startxt + \xshiftera - .575,  \verlinetwo + \yshiftera + .91) -- (\startxt + \xshiftera - .63,  \verlinetwo + \yshiftera + 1.16) -- (\startxt + \xshiftera - .66,  \verlinetwo + \yshiftera + 1.5) -- (\startxt - \xshiftera - .66,  \verlinetwo - \yshiftera + 0.2 + 1.5) -- (\startxt - \xshiftera - .63,  \verlinetwo - \yshiftera + 0.2 + 1.16) -- (\startxt - \xshiftera - .575,  \verlinetwo - \yshiftera + 0.2 + .91) -- (\startxt - \xshiftera - .5, \verlinetwo - \yshiftera + 0.2 + .71) -- (\startxt - \xshiftera - .44, \verlinetwo - \yshiftera + 0.2 + .55) -- (\startxt - \xshiftera, \verlinetwo - \yshiftera + 0.2 -0.12) -- cycle;

\draw[black, join = round, opacity=0.2, fill = black]				 (\startxt + \xshiftera, \verlinetwo + \yshiftera -0.12) -- (\startxt + \xshiftera - .44, \verlinetwo + \yshiftera + .55) -- (\startxt + \xshiftera - .5, \verlinetwo + \yshiftera + .71)  -- (\startxt + \xshiftera - .575,  \verlinetwo + \yshiftera + .91) -- (\startxt + \xshiftera - .63,  \verlinetwo + \yshiftera + 1.16) -- (\startxt + \xshiftera - .66,  \verlinetwo + \yshiftera + 1.5) -- (\startxt - \xshiftera - .66,  \verlinetwo - \yshiftera + 0.2 + 1.5) -- (\startxt - \xshiftera - .63,  \verlinetwo - \yshiftera + 0.2 + 1.16) -- (\startxt - \xshiftera - .575,  \verlinetwo - \yshiftera + 0.2 + .91) -- (\startxt - \xshiftera - .5, \verlinetwo - \yshiftera + 0.2 + .71) -- (\startxt - \xshiftera - .44, \verlinetwo - \yshiftera + 0.2 + .55) -- (\startxt - \xshiftera, \verlinetwo - \yshiftera + 0.2 -0.12) -- cycle;

\end{tikzpicture}
\caption{The moment map image $\mu(\SSS^2 \times \SSS^2)$, drawn twice. The figure on the left shows the similarity with that of the standard action of $\mathbb{T}^2$ on $\SSS^2 \times \SSS^2$ from classic symplectic geometry and the one on the right is bent to be visually similar to Figure \ref{fig:bline}.}
\label{fig:momentmapS2S2}
\end{figure}
\end{example}

\begin{example}\label{ex:mmt2}
Consider the $b$-symplectic manifold
\[
(\TT^2 = \{(\theta_1, \theta_2) \in (\mathbb{R} / 2\pi)^2\}, Z = \{\theta_1 \in \{0, \pi\}\}, \omega=\frac{d \theta_1}{\sin \theta_1}\wedge d\theta_2)
\]
with $\SSS^1$-action given by the flow of $\frac{\partial}{\partial \theta_2}$. Let $X \in \mathfrak{t}$ be the element satisfying $X^{\#} = \frac{\partial}{\partial \theta_2}$. Let $G$ be the cycle with two edges, and let $w$ map one edge to $X^*$ and the other to $-X^*$. A (smooth) moment map for the $\SSS^1$-action is given by extending
\[
\mathbb{T}^2 \backslash Z \rightarrow \mathfrak{t}^*, \text{\,\,\,\,\,\,\,}
(\theta_1, \theta_2) \mapsto
\left\{ \begin{array}{r l}

\left(1, \textrm{log}\left| \frac{1 + \cos\theta_1}{\sin\theta_1}\right|\right) & \text{if } 0 < \theta_1 < \pi \\

\left(0, \textrm{log}\left| \frac{1 + \cos\theta_1}{\sin\theta_1}\right|\right) & \text{if } \pi < \theta_1 <2\pi
\end{array}\right.
\]
to $\mu: \TT^2 \rightarrow \mathcal{R}_{\mathcal{G}}$, as shown in Figure \ref{fig:momentmapT2}.

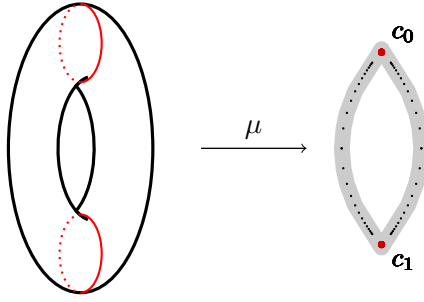
\begin{figure}[ht]
\begin{tikzpicture}[scale = 0.8]

\pgfmathsetmacro{\sizer}{2}
\pgfmathsetmacro{\donutcenterx}{0}
\pgfmathsetmacro{\donutcentery}{0}
\pgfmathsetmacro{\basept}{5}

\DrawDonut{\donutcenterx}{\donutcentery}{.6*\sizer}{1.2*\sizer}{-90}{black}{very thick}
\draw [red, thick] (\donutcenterx, \donutcentery - .55*\sizer) arc (90:-90:0.35 and .65);
\draw [red, thick, dotted] (\donutcenterx, \donutcentery - .55*\sizer) arc (90:270:0.35 and .65);

\draw [red, thick] (\donutcenterx, \donutcentery + .55*\sizer) arc (-90:90:0.35 and .65);
\draw [red, thick, dotted] (\donutcenterx, \donutcentery + .55*\sizer) arc (270:90:0.35 and .65);

\foreach \paramt in {-3, -2.75, ..., 0.01}	
	{
	\pgfmathsetmacro{\height}{1.5 * exp(\paramt)}	
	
	\pgfmathsetmacro{\xshift}{0.25*(\height * (\height - 3)) - 0.1}
	\draw[fill = black] (\basept + \xshift, \height - 1.5) circle(.1mm);
	\draw[fill = black] (\basept - \xshift, \height - 1.5) circle(.1mm);
	\draw[red, fill = red] (\basept, -1.6) circle(.5mm) node[below right, black] {\small{$c_1$}};
	
	\draw[fill = black] (\basept + \xshift, 1.5 - \height) circle(.1mm);
	\draw[fill = black] (\basept - \xshift, 1.5 - \height) circle(.1mm);
	\draw[red, fill = red] (\basept, 1.6) circle(.5mm) node[above right, black] {\small{$c_0$}};
	}

\draw (2, 0) edge node[above] {$\mu$} (3.75, 0);
\draw[->] (2, 0) -- (3.75, 0);

\draw[line width = 6pt, join = round, opacity=0.2] (\basept - .66, 0) -- (\basept - .63,  1.16 - 1.5) -- (\basept - .575,  .91 - 1.5) -- (\basept - .5, .71-1.5) -- (\basept - .44, .55-1.5) -- (\basept, -1.65) -- (\basept + .44, .55-1.5) -- (\basept + .5, .71-1.5) -- (\basept + .575,  .91 - 1.5) -- (\basept + .63,  1.16 - 1.5) -- (\basept + .66, 0) -- (\basept + .63,  1.5 - 1.16) -- (\basept + .575,  1.5 - .91) -- (\basept + .5, 1.5 - .71) -- (\basept + .44, 1.5 - .55) -- (\basept, 1.65) -- (\basept - .44, 1.5 - .55) -- (\basept - .5, 1.5-.71) -- (\basept - .575,  1.5-.91) -- (\basept - .63,  1.5-1.16) -- cycle;

\end{tikzpicture}
\caption{The moment map $\mu: \TT \rightarrow \mathcal{R}_{\mathcal{G}}$.}
\label{fig:momentmapT2}
\end{figure}
\end{example}

\section{The moment map of a toric action on a $b$-symplectic manifold}\label{sec:mmallofM}

The main result of this section is Theorem \ref{thm:mmexistence}, which states that every $b$-symplectic manifold with a toric action has a globally-defined moment map. We begin by studying a toric action in a neighborhood of $Z$ .

First, we note that Proposition \ref{prop:localmomentmap} can now be rewritten in terms of the new notion of moment map. The weighted adjacency graph of the neighborhood $\mathcal{U}$ of $Z$ is $(G, w)$, where $G$ consists of one edge $e$ connecting two vertices, and $w(e)$ is equal to the modular weight $v_Z$ of $Z$. The classical moment map in equation (\ref{eqn:localmm}) can be rewritten as a $b$ moment map, now defined in the whole open set $\mathcal{U}$:
\begin{equation}\label{eqn:newlocalmm}
\mu: \mathcal{L} \times \SSS^1 \times (-\varepsilon, \varepsilon) \rightarrow \mathcal{R}_\mathcal{G} \cong \mathfrak{t}_{Z}^* \times \mathbb{R},\text{\,\,\,\,\,\,\,\,\,}
(\ell, \rho, t) \mapsto (\mu_{\mathcal{L}}(\ell), t),
\end{equation}
where the isomorphism $\mathcal{R}_{\mathcal{G}} \cong \mathfrak{t}_{Z}^* \times \mathbb{R}$ is the map $y_{X, e}$ described in the definition of the $b$ moment codomain, for a choice of $X\in \mathfrak{t}$ that represents a generator of $\mathfrak{t} / \mathfrak{t}_Z$ and pairs positively with $v_{Z}$.

Secondly, we need a proposition that describes a local model for the $b$-symplectic manifold in a neighborhood of $Z$.

\begin{proposition}\label{prop:local_model}(Local Model)
Let $G$ be a graph with two vertices and a single edge $e$, let $w$ be any weight function, and let $X$ be any lattice generator of $\mathfrak{t}/\mathfrak{t}_w$ that pairs positively with $w(e)$. Recall that the function $y_{X, e}$ gives an isomorphism between $\mathcal{R}_{\mathcal{G}}$ and $\mathfrak{t}_w^* \times \mathbb{R}$. Then for any Delzant polytope $\Delta \subseteq \mathfrak{t}_w^*$ with corresponding symplectic toric manifold $(X_{\Delta}, \omega_{\Delta}, \mu_{\Delta})$, define the \textbf{local model} $b$-symplectic manifold as
\[
M_{\textrm{loc}} = X_{\Delta} \times \SSS^1 \times \mathbb{R} \hspace{1cm}
\omega_{\textrm{loc}} = \omega_{\Delta} + c\frac{dt}{t} \wedge d\theta
\]
where $\theta$ and $t$ are the coordinates on $\SSS^1$ and $\mathbb{R}$ respectively. The $\SSS^1 \times \mathbb{T}_Z$ action on $M_{\textrm{loc}}$ given by $(\rho, g) \cdot (x, \theta, t) = (g \cdot x, \theta + \rho, t)$ has moment map $\mu_{\textrm{loc}}(x, \theta, t) = (\mu_{\Delta}(x), t) \in \mathfrak{t}_w^* \times \mathbb{R} \cong \mathcal{R}_{\mathcal{G}}$.

For any toric action on a $b$-manifold $(M, Z, \omega)$ with moment map $\mu: M \rightarrow \mathcal{R}_{\mathcal{G}}$ such that $\mu(U) = \Delta \times (-\epsilon \leq y_0 \leq \epsilon)$ in a neighborhood $U$ of $Z$, there is an equivariant $b$-symplectomorphism $\varphi: M_{\textrm{loc}} \rightarrow M$ in a neighborhood of $X_{\Delta} \times \SSS^1 \times \{0\}$ satisfying $\mu\circ \varphi = \mu_{\textrm{loc}}$.
\end{proposition}
\begin{proof}
Fix a symplectic leaf $\mathcal{L} \subseteq Z$. Observe that $\textrm{im}(\restr{\mu}{\mathcal{L}}) = \{0\} \times \Delta$, and define $\mu_{\mathcal{L}}: \mathcal{L} \rightarrow \mathfrak{t}_Z^*$ to be the projection of $\restr{\mu}{\mathcal{L}}$ to its second coordinate. By the classic Delzant theorem there is an equivariant symplectomorphism $\varphi_{\Delta}: (X_{\Delta}, \omega_{\Delta}) \rightarrow (\mathcal{L}, \omega_{\mathcal{L}})$ such that $\mu_{\Delta} = \mu_{\mathcal{L}} \circ \varphi_{\Delta}$. As in the proof of Proposition \ref{prop:localmomentmap}, let $f$ be a primitive of $\iota_{X^{\#}}\omega$ and $y$ a defining function for $Z$ with the property that $f = c\textrm{log}|y|$ near $Z$. Let $u$ be a $\mathbb{T}^n$-equivariant vector field in a neighborhood of $Z$, such that $dy(u) = 1$ and $\iota_{Y^{\#}}\omega(u) = 0$ for all $Y \in \mathfrak{t}_Z$. Then the map
$$\varphi: M_{\textrm{loc}} = X_{\Delta} \times \SSS^1 \times \mathbb{R} \rightarrow M,\text{\,\,\,\,\,\,}
(x, \theta, t) \mapsto \Phi^u_t \circ \Phi^{X^{\#}}_{\theta} \circ \varphi_{\Delta}(x)
$$
is defined in a neighborhood of $X_{\Delta} \times \SSS^1 \times \{0\}$.

It follows by the equivariance of $u, X^{\#}$, and $\varphi_{\Delta}$ that $\varphi$ itself is equivariant. Next, observe that
\[
\mu \circ \varphi(x, \theta, t) = \mu \circ \Phi_{\theta}^{X^{\#}} \circ \Phi_t^u \circ \varphi_{\Delta}(x) = \mu \circ \Phi_t^u \circ \varphi_{\Delta}(x)
\]
since $\mu$ is $\mathbb{T}^n$-invariant, and that the $\mathfrak{t}_Z^*$-component of $\mu \circ \Phi_t^u \circ \varphi_{\Delta}(x)$ will be $\varphi_{\Delta}(x)$, since $\iota_{Y^{\#}}\omega(u) = 0$ for all $Y \in \mathfrak{t}_Z$. The $\mathbb{R}$-component of $\mu \circ \Phi_t^u \circ \varphi_{\Delta}(x)$ will be $t$, since the vector field $u$ satisfies $dy(u) = 1$. Therefore, $\mu \circ \varphi = \mu_{\textrm{loc}}$.

This shows that on $(M_{\textrm{loc}}, \omega_{\textrm{loc}})$ and $(M_{\textrm{loc}}, \varphi^*(\omega))$, the same moment map $\mu_{\textrm{loc}}$ corresponds to the same action. Our next goal is to show that $\restr{\varphi^*\omega}{Z} = \restr{\omega_{\textrm{loc}}}{Z}$. For $z \in Z$, let $A \subseteq {^b}T_zM$ be the symplectic orthogonal to $(X^{\#})_z$. Restriction of the canonical map ${^b}T_zM \rightarrow T_zM$ to $A$ leaves its image unchanged (since the kernel of the canonical map, $\mathbb{L}$, is not in $A$). By picking a basis for $T_zL \subseteq T_zZ$ and pulling it back to $A$, and then adding $(X^{\#})_z$ and $(t\frac{\partial}{\partial t})_z$, we obtain a basis of ${^b}T_zZ$. By calculating the value of $\omega_z$ with respect to this basis, and using the facts that $\varphi_{\Delta}$ is a symplectomorphism and that
\[
\varphi^*\omega(t\frac{\partial}{\partial t}, \frac{\partial}{\partial \theta}) = \omega(yu, X^{\#}) = d(c\log|y|)(yu) = c,
\]
we conclude that $\restr{\varphi^*\omega}{Z} = \restr{\omega_{\textrm{loc}}}{Z}$. Finally, we apply Moser's path method to construct a symplectomorphism between $\varphi^*\omega$ and $\omega_{\textrm{loc}}$.

Note that $\varphi^*\omega - \omega_{\textrm{loc}}$ is $\TT^n$-invariant and has the property that the tangent space to each orbit is in the kernel of $\varphi^*\omega - \omega_{\textrm{loc}}$. Therefore, we can write $\varphi^*\omega - \omega_{\textrm{loc}}$ as the pullback under $\mu_{\textrm{loc}}$ of a smooth form $\nu$ on $\mathcal{R}_{\mathcal{G}}$. Let $\alpha$ be the pullback (under $\mu_{\textrm{loc}}$) of a primitive of $\nu$. Then $\alpha$ is a primitive of $\varphi^*\omega - \omega_{\textrm{loc}}$ with the property that the vector fields defined using Moser's path method will be tangent to the orbits of the torus action, and also with the property that $\alpha$ is torus invariant. Therefore, the equivariant symplectomorphism it defines leaves the moment map unchanged, completing the proof.
\end{proof}

\begin{theorem}\label{thm:mmexistence}
Let $(M, Z, \omega, \mathbb{T}^n)$ be a $b$-symplectic manifold with an effective Hamiltonian toric action, and let $\mathcal{G} = (G, w)$ be its weighted adjacency graph. Then there is a moment map $\mu: M \rightarrow \mathcal{R}_{\mathcal{G}}$.
\end{theorem}
\begin{proof}
We first consider the case when the graph is a line: for $1 \leq i \leq N$, the edge $e_i$ connects vertex $v_{i-1}$ to vertex $v_i$. Let $W_i$ be the component of $M \backslash Z$ corresponding to $v_i$, and let $Z_i$ be the component of $Z$ corresponding to $e_i$. Fix any moment map $\mu_{W_0}: W_0 \rightarrow \mathfrak{t}^*$ for the action on $W_0$. Identifying the codomain $\mathfrak{t}^*$ with $\{v_0\} \times \mathfrak{t}^* \subseteq \mathcal{R}_{\mathcal{G}}$ gives a moment map $\mu_{W_0}: W_0 \rightarrow \mathcal{R}_{\mathcal{G}}$. By Proposition \ref{prop:localmomentmap}, there is a moment map $\mu_{U_1}$ for the $\mathbb{T}^n$-action in a neighborhood $U_1$ of $Z_1$. The two moment maps
\[
\restr{\mu_{W_0}}{W_0 \cap U_1} \ \ \textrm{and} \ \ \restr{\mu_{U_1}}{W_0 \cap U_1}
\]
correspond to the same $\mathbb{T}^n$-action on $W_0 \cap U_1$, so after a translation we may glue $\mu_{W_0}$ and $\mu_{U_1}$ into a moment map defined on all of $W_0 \cup U_1$. We continue extending the moment map in this manner until it is defined on all of $M$.

When the adjacency graph is a cycle, the above construction breaks down in the final stage; after choosing the correct translation of the moment map $\mu_{U_{N}}$ so that it agrees with $\mu_{W_{N-1}}$ on the overlap of their domains, it may not be the case that $\mu_{U_{N}}$ agrees with $\mu_{W_0}$ on the overlap of their domains. Pick some $p \in U_N \cap W_0$, and define
\[
x_{\textrm{start}} = \mu_{W_0}(p) \ \ \textrm{and}  \ \ x_{\textrm{end}} = \mu_{U_N}(p)
\]
and assume without loss of generality that $x_{\textrm{start}} = (v_0, 0) \in \mathcal{R}_{\mathcal{G}}$. Let $\gamma: \SSS^1 = \mathbb{R} / \mathbb{Z} \rightarrow M$ be a loop with $\gamma(0) = \gamma(1) = p$ that visits the sets $W_0, U_1, W_1, \dots, W_{N-1}, U_N, W_N = W_0$ in that order. Then, for any $X \in \mathfrak{t}$, we have
\[
x_{\textrm{end}} = (v_0, x) \ \ \textrm{where} \ \mu^{X}(x) = \left\{ \begin{array}{c l} \int_\gamma \iota_{X^{\#}}\omega & X \in \mathfrak{t}_Z\\ \bint{\gamma}{} \iota_{X^{\#}}\omega & X \notin \mathfrak{t}_Z \end{array} \right.
\]
When $X \in \mathfrak{t}_Z$, the 1-form $\iota_{X^{\#}}\omega$ has a smooth primitive, so the above integral equals zero. When $X \notin \mathfrak{t}_Z$, the 1-form $\iota_{X^{\#}}\omega$ has a ${^b}C^{\infty}$ primitive, and since the Liouville volume of the pullback is still zero the above integral equals zero. Therefore, $x_{\textrm{end}} = (v_0, 0)$ and the moment maps for the sets $W_i$ and $U_i$ glue into a moment map $\mu: M \rightarrow \mathcal{R}_{\mathcal{G}}$.
\end{proof}
Theorem \ref{thm:mmexistence} proves that every Hamiltonian toric action has a moment map. However, note that as in the classic case different translations of the moment map correspond to the same action.

\section{Delzant theorem}

In this section, we prove a Delzant theorem in $b$-geometry. We begin by defining the notion of a $b$-symplectic toric manifold and of a Delzant $b$-polytope.

\begin{definition}\label{def:bstm}
A \textbf{$b$-symplectic toric manifold} is
\[
(M^{2n}, Z, \omega, \mu: M \rightarrow \mathcal{R}_{\mathcal{G}})
\]
where $(M, Z, \omega)$ is a $b$-symplectic manifold and $\mu$ is a moment map for a toric action on $(M, Z, \omega)$.
\end{definition}

Notice that the definition of a $b$-symplectic toric manifold also implicitly includes the information of $\mathcal{G}$. The definition of a polytope in $\mathcal{R}_{\mathcal{G}}$ will use the definition of a half-space; it might be helpful to look at the examples of half-spaces in $\mathcal{R}_{\mathcal{G}}$  in Figure \ref{fig:halfspaces} before reading the formal definition. Although the boundaries of a half-spaces in $\mathcal{R}_{\mathcal{G}}$ are straight lines when restricted to each $\{k\} \times \mathfrak{t}^* \subseteq \mathcal{R}_{\mathcal{G}}$, they appear curved in Figure \ref{fig:halfspaces} because of the way $\mathcal{R}_{\mathcal{G}}$ is drawn. Notice that the boundary of a half-space will not intersect $\mathcal{Z}_{\mathcal{G}}$ unless it is perpendicular to it.

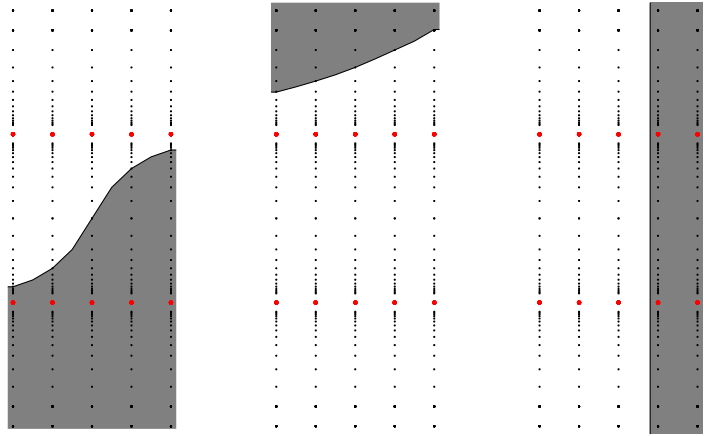
\begin{figure}[ht]
\centering
\begin{tikzpicture}[scale = 0.7]

\pgfmathsetmacro{\startxone}{0}
\pgfmathsetmacro{\startxtwo}{5}
\pgfmathsetmacro{\startxthree}{10}

\pgfmathsetmacro{\startxt}{6} \pgfmathsetmacro{\verlinebottom}{0}

\fill[gray] (-1.6, 1.5 + 0.1 - 1.3) -- (-1.5, 1.5 + 0.1 - 1.3) -- (-1.125, 1.5 + 0.1 - 1.17) -- (-.75, 1.5 + 0.1 - .948) -- (-0.375, 1.5 + 0.1 - .59) -- (0, 1.5 + 0.1) -- (0.375, 1.5 + 0.1 + .59) -- (.75, 1.5 + 0.1 + .948) -- (1.125, 1.5 + 0.1 + 1.17) -- (1.5, 1.5 + 0.1 + 1.3)  -- (1.6, 1.5 + 0.1 + 1.3) -- (1.6, -1.5 - 0.1 - 0.5 - 0.3) -- (-1.6, -1.5 - 0.1 - 0.5 - 0.3) -- cycle;
\draw (-1.6, 1.5 + 0.1 - 1.3) -- (-1.5, 1.5 + 0.1 - 1.3) -- (-1.125, 1.5 + 0.1 - 1.17) -- (-.75, 1.5 + 0.1 - .948) -- (-0.375, 1.5 + 0.1 - .59) -- (0, 1.5 + 0.1) -- (0.375, 1.5 + 0.1 + .59) -- (.75, 1.5 + 0.1 + .948) -- (1.125, 1.5 + 0.1 + 1.17) -- (1.5, 1.5 + 0.1 + 1.3)  -- (1.6, 1.5 + 0.1 + 1.3);

\fill[gray] (\startxtwo -1.6, 3.2 + 1.5 + 0.1 - 0.8) -- (\startxtwo -1.5, 3.2 + 1.5 + 0.1 - 0.8) -- (\startxtwo -1.125, 3.2 + 1.5 + 0.1 - .7) -- (\startxtwo -.75, 3.2 + 1.5 + 0.1 - .59) -- (\startxtwo -0.375, 3.2 + 1.5 + 0.1 - .47) -- (\startxtwo +0, 3.2 + 1.5 + 0.1 - .33) -- (\startxtwo +0.375, 3.2 + 1.5 + 0.1 -.17) -- (\startxtwo +.75, 3.2 + 1.5 + 0.1 + 0) -- (\startxtwo +1.125, 3.2 + 1.5 + 0.1 + .17) -- (\startxtwo +1.5, 3.2 + 1.5 + 0.1 + .39)  -- (\startxtwo +1.6, 3.2 + 1.5 + 0.1 + .39) -- (\startxtwo +1.6, 5.7) -- (\startxtwo -1.6, 5.7) -- cycle;
\draw (\startxtwo -1.6, 3.2 + 1.5 + 0.1 - 0.8) -- (\startxtwo -1.5, 3.2 + 1.5 + 0.1 - 0.8) -- (\startxtwo -1.125, 3.2 + 1.5 + 0.1 - .7) -- (\startxtwo -.75, 3.2 + 1.5 + 0.1 - .59) -- (\startxtwo -0.375, 3.2 + 1.5 + 0.1 - .47) -- (\startxtwo +0, 3.2 + 1.5 + 0.1 - .33) -- (\startxtwo +0.375, 3.2 + 1.5 + 0.1 -.17) -- (\startxtwo +.75, 3.2 + 1.5 + 0.1 + 0) -- (\startxtwo +1.125, 3.2 + 1.5 + 0.1 + .17) -- (\startxtwo +1.5, 3.2 + 1.5 + 0.1 + .39)  -- (\startxtwo +1.6, 3.2 + 1.5 + 0.1 + .39);

\fill[gray] (\startxthree + 0.6, 5.7) -- (\startxthree + 0.6, -2.5) -- (\startxthree + 1.6, -2.5) -- (\startxthree + 1.6, 5.7) -- cycle;
\draw (\startxthree + 0.6, 5.7) -- (\startxthree + 0.6, -2.5);

\foreach \horizshift in {-2, -1, 0, 1, 2}{
\foreach \paramt in {-3, -2.75, ..., 0.01}	
	{
	\pgfmathsetmacro{\height}{1.5 * exp(\paramt)}
	\pgfmathsetmacro{\squeezedhorizshift}{0.75 * \horizshift}
	\draw[fill = black] (\startxone + \squeezedhorizshift, \verlinebottom + \height + 0.1) circle(.1mm);
	\draw[fill = black] (\startxone + \squeezedhorizshift, \verlinebottom - \height - 0.1) circle(.1mm);
	\draw[fill = black] (\startxone + \squeezedhorizshift, \verlinebottom + 3.2 + \height + 0.1) circle(.1mm);
	\draw[fill = black] (\startxone + \squeezedhorizshift, \verlinebottom + 3.2 - \height - 0.1) circle(.1mm);
	
	\draw[fill = black] (\startxtwo + \squeezedhorizshift, \verlinebottom + \height + 0.1) circle(.1mm);
	\draw[fill = black] (\startxtwo + \squeezedhorizshift, \verlinebottom - \height - 0.1) circle(.1mm);
	\draw[fill = black] (\startxtwo + \squeezedhorizshift, \verlinebottom + 3.2 + \height + 0.1) circle(.1mm);
	\draw[fill = black] (\startxtwo + \squeezedhorizshift, \verlinebottom + 3.2 - \height - 0.1) circle(.1mm);
	
	\draw[fill = black] (\startxthree + \squeezedhorizshift, \verlinebottom + \height + 0.1) circle(.1mm);
	\draw[fill = black] (\startxthree + \squeezedhorizshift, \verlinebottom - \height - 0.1) circle(.1mm);
	\draw[fill = black] (\startxthree + \squeezedhorizshift, \verlinebottom + 3.2 + \height + 0.1) circle(.1mm);
	\draw[fill = black] (\startxthree + \squeezedhorizshift, \verlinebottom + 3.2 - \height - 0.1) circle(.1mm);

	\foreach \extratail in {0.25, 0.5}
		{
		\pgfmathsetmacro{\extrat}{\extratail * 1.5}
		\draw[fill = black] (\startxone + \squeezedhorizshift, \verlinebottom + 3.2 + 0.1 + 1.5 + \extrat) circle(.1mm);
		\draw[fill = black] (\startxone + \squeezedhorizshift, \verlinebottom -0.1 - 1.5 - \extrat) circle(.1mm);		
		
		\draw[fill = black] (\startxtwo + \squeezedhorizshift, \verlinebottom + 3.2 + 0.1 + 1.5 + \extrat) circle(.1mm);
		\draw[fill = black] (\startxtwo + \squeezedhorizshift, \verlinebottom -0.1 - 1.5 - \extrat) circle(.1mm);		
		
		\draw[fill = black] (\startxthree + \squeezedhorizshift, \verlinebottom + 3.2 + 0.1 + 1.5 + \extrat) circle(.1mm);
		\draw[fill = black] (\startxthree + \squeezedhorizshift, \verlinebottom -0.1 - 1.5 - \extrat) circle(.1mm);		
		}
	
	\draw[red, fill = red] (\startxone + \squeezedhorizshift, \verlinebottom) circle(.3mm);
	\draw[red, fill = red] (\startxone + \squeezedhorizshift, \verlinebottom + 3.2) circle(.3mm);
	
	\draw[red, fill = red] (\startxtwo + \squeezedhorizshift, \verlinebottom) circle(.3mm);
	\draw[red, fill = red] (\startxtwo + \squeezedhorizshift, \verlinebottom + 3.2) circle(.3mm);
	
	\draw[red, fill = red] (\startxthree + \squeezedhorizshift, \verlinebottom) circle(.3mm);
	\draw[red, fill = red] (\startxthree + \squeezedhorizshift, \verlinebottom + 3.2) circle(.3mm);
	
	}
	}
	
\end{tikzpicture}
\caption{Examples of half-spaces in $\mathcal{R}_{\mathcal{G}}$.}
\label{fig:halfspaces}
\end{figure}

\begin{definition}\label{def:halfspace} Consider the two following kinds of hypersurfaces in $\mathcal{R}_{\mathcal{G}}$, where $X \in \mathfrak{t}$, $Y \in \mathfrak{t}_w$, $k \in \mathbb{R}$ and $v$ is a vertex of $G$:
\begin{align*}
&A_{X, k, v} = \{(v, \xi) \mid \langle \xi, X \rangle = k\} \subseteq \{v\} \times \mathfrak{t}^* \subseteq \mathcal{R}_{\mathcal{G}},\\
&B_{Y, k} = \overline{\{(v, \xi) \mid \langle \xi, Y \rangle = k, v \ \textrm{a vertex of} \ G\}} \subseteq \overline{\mathcal{R}_{\mathcal{G}} \backslash \mathcal{Z}_{\mathcal{G}}} = \mathcal{R}_{\mathcal{G}}.
\end{align*}

When $G$ is a line, the complement of any hypersurface in $\mathcal{R}_{\mathcal{G}}$ is two connected components. The closure of any such component is a \textbf{half-space} in $\mathcal{R}_{\mathcal{G}}$. When $G$ is a cycle, the hypersurfaces of type $A_{X, k, c}$ do not separate the space, so only the closure of a connected component of the complement of some $B_{X, k} \subseteq \mathcal{R}_{\mathcal{G}}$ is called a \textbf{half-space}.
\end{definition}
In Figure \ref{fig:halfspaces}, the first two images are half-spaces of type $A_{X, k, c}$, while the rightmost image is a half-space of type $B_{X, k}$.
\begin{definition}
A \textbf{$b$-polytope} in $\mathcal{R}_{\mathcal{G}}$ is a bounded subset $P$ that intersects each component of $Z_{\mathcal{G}}$ and can be expressed as a finite intersection of half-spaces.
\end{definition}
If the condition that $P$ must intersect each component of $Z_{\mathcal{G}}$ were removed from the definition of a polytope, then for any $\mathcal{G} = (G, w)$ and $\mathcal{G}' = (G', w')$ such that $G$ is a subgraph of $G'$ and $w'$ is an extension of $w$, any polytope in $\mathcal{R}_{\mathcal{G}}$ would also be a polytope in $\mathcal{R}_{\mathcal{G}'}$. The upcoming statement of Theorem \ref{thm:bDelzant}, which generalizes the Delzant theorem, is easier to state when we disallow the possibility of having such extraneous parts of the $b$ moment codomain.

\begin{example}\label{examp:bpolytopes}
Figure \ref{fig:polytopes} shows two examples of $b$-polytopes. In both cases, the torus has dimension two. The $b$ moment codomain on the left corresponds to a graph with an edge and two vertices, while the $b$ moment codomain on the right corresponds to the cycle on two vertices (the top of the picture on the right is identified with the bottom of the picture).
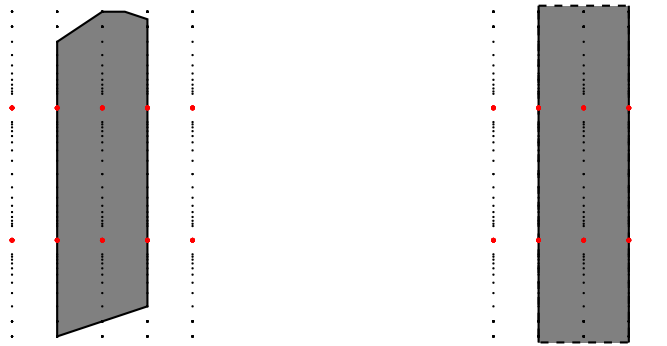
\begin{figure}[ht]
\centering
\begin{tikzpicture}[scale = 0.8]

\pgfmathsetmacro{\startxone}{0}
\pgfmathsetmacro{\startxtwo}{8}

\pgfmathsetmacro{\startxt}{6} \pgfmathsetmacro{\verlinebottom}{0}

\draw [thick, fill = gray] (\startxone - 0.75, \verlinebottom + 2.2 + 0.1 + 1) -- (\startxone, \verlinebottom + 2.2 + 0.1 + 1.5)  -- (\startxone + 0.375, \verlinebottom + 2.2 + 0.1 + 1.5) -- (\startxone + 0.75, \verlinebottom + 2.2 + 0.1 + 1.375) -- (\startxone + 0.75, \verlinebottom - 0.1 - 1) 	-- (\startxone - 0.75, \verlinebottom - 0.1 - 1.5) -- cycle;

\fill [gray] (\startxtwo - 0.75, \verlinebottom + 2.2 + 0.1 + 1.6) -- (\startxtwo + 0.75, \verlinebottom + 2.2 + 0.1 + 1.6)  -- (\startxtwo + 0.75, \verlinebottom - 0.1 - 1.6) 	-- (\startxtwo - 0.75, \verlinebottom - 0.1 - 1.6) -- cycle;

\draw[thick, dashed] (\startxtwo - 0.75, \verlinebottom + 2.2 + 0.1 + 1.6) -- (\startxtwo + 0.75, \verlinebottom + 2.2 + 0.1 + 1.6)  -- (\startxtwo + 0.75, \verlinebottom - 0.1 - 1.6) 	-- (\startxtwo - 0.75, \verlinebottom - 0.1 - 1.6) -- cycle;

\draw[thick] (\startxtwo + 0.75, \verlinebottom + 2.2 + 0.1 + 1.6)  -- (\startxtwo + 0.75, \verlinebottom - 0.1 - 1.6);
\draw[thick] (\startxtwo - 0.75, \verlinebottom + 2.2 + 0.1 + 1.6)  -- (\startxtwo - 0.75, \verlinebottom - 0.1 - 1.6);

\foreach \horizshift in {-2, -1, 0, 1, 2}{
\foreach \paramt in {-2, -1.75, ..., 0.01}	
	{
	\pgfmathsetmacro{\height}{1 * exp(\paramt)}
	\pgfmathsetmacro{\squeezedhorizshift}{0.75 * \horizshift}
	\draw[fill = black] (\startxone + \squeezedhorizshift, \verlinebottom + \height + 0.1) circle(.1mm);
	\draw[fill = black] (\startxone + \squeezedhorizshift, \verlinebottom - \height - 0.1) circle(.1mm);
	\draw[fill = black] (\startxone + \squeezedhorizshift, \verlinebottom + 2.2 + \height + 0.1) circle(.1mm);
	\draw[fill = black] (\startxone + \squeezedhorizshift, \verlinebottom + 2.2 - \height - 0.1) circle(.1mm);
	
	\draw[fill = black] (\startxtwo + \squeezedhorizshift, \verlinebottom + \height + 0.1) circle(.1mm);
	\draw[fill = black] (\startxtwo + \squeezedhorizshift, \verlinebottom - \height - 0.1) circle(.1mm);
	\draw[fill = black] (\startxtwo + \squeezedhorizshift, \verlinebottom + 2.2 + \height + 0.1) circle(.1mm);
	\draw[fill = black] (\startxtwo + \squeezedhorizshift, \verlinebottom + 2.2 - \height - 0.1) circle(.1mm);		
	
	\foreach \extratail in {0.25, 0.5}
		{
		\pgfmathsetmacro{\extrat}{\extratail * 1}
		\draw[fill = black] (\startxone + \squeezedhorizshift, \verlinebottom + 2.2 + 0.1 + 1 + \extrat) circle(.1mm);
		\draw[fill = black] (\startxone + \squeezedhorizshift, \verlinebottom -0.1 - 1 - \extrat) circle(.1mm);		
		
		\draw[fill = black] (\startxtwo + \squeezedhorizshift, \verlinebottom + 2.2 + 0.1 + 1 + \extrat) circle(.1mm);
		\draw[fill = black] (\startxtwo + \squeezedhorizshift, \verlinebottom -0.1 - 1 - \extrat) circle(.1mm);		
		}
	
	\draw[red, fill = red] (\startxone + \squeezedhorizshift, \verlinebottom) circle(.3mm);
	\draw[red, fill = red] (\startxone + \squeezedhorizshift, \verlinebottom + 2.2) circle(.3mm);
	
	\draw[red, fill = red] (\startxtwo + \squeezedhorizshift, \verlinebottom) circle(.3mm);
	\draw[red, fill = red] (\startxtwo + \squeezedhorizshift, \verlinebottom + 2.2) circle(.3mm);

	}
	}
	
\end{tikzpicture}
\caption{Examples of $b$-polytopes.}
\label{fig:polytopes}
\end{figure}
\end{example}

The definitions of many features of classic polytopes, such as \textbf{facets}, \textbf{edges}, and \textbf{vertices}, generalize in a natural way to $b$-polytopes, as does the notion of a \textbf{rational} polytope (one in which the $X$'s and $k$'s in Definition \ref{def:halfspace} are rational). We state some properties of $b$-polytopes, all of which are straightforward consequences of the definition.

\begin{itemize}
\item Because $P$ must intersect each component of $Z_{\mathcal{G}}$, the only hypersurfaces of type $A_{X, k, v}$ that will appear as boundaries of half-spaces constituting $P$ will be have $c = a-1$ or $c = N$.

\item No vertex of $P$ lies on $Z_{\mathcal{G}}$.

\item Given a polytope $P \subseteq \mathcal{R}_{\mathcal{G}}$, there is a (classic) polytope $\Delta_Z \subseteq \mathcal{Z}_{\mathcal{G}}$ having the property that the intersection of $P$ with each component of $\mathcal{Z}_{\mathcal{G}}$ is $\Delta_Z$.

\item $P$ is locally isomorphic to $\{-\varepsilon \leq y_i \leq \varepsilon\} \times \Delta_Z$ near each component of $\mathcal{Z}_{\mathcal{G}}$, and is isomorphic to $\Delta_Z \times \mathbb{R}$ in any component $\{v\} \times \mathfrak{t}^* \subseteq \mathcal{R}_{\mathcal{G}}$ except those vertices $v$ which are leaves of $G$.

\item When $v$ is a leaf of $G$, the restriction of $P$ to $\{v\} \times \mathfrak{t}^*$ is a polyhedron with recession cone\footnote{The recession cone of a convex set $A \subseteq V$ is $recc(A)=\{v\in V\mid \forall_{a\in A} \,a+v\in A\}$.} equal to $\mathbb{R}_0^- w(e)$, where $e=(v,v')$ is an edge.
\end{itemize}

\noindent Because no vertex of $P$ lies on $\mathcal{Z}_{\mathcal{G}}$, the definition of a Delzant polytope generalizes easily to the context of $b$-polytopes.

\begin{definition}
When $G$ is a line, a $b$-polytope $P \subseteq \mathcal{R}_{\mathcal{G}}$ is \textbf{Delzant} if for every vertex $v$ of $P$, there is a lattice basis $\{u_i\}$ of $\mathfrak{t}^*$ such that the edges incident to $v$ can be written near $v$ in the form $v + tu_i$ for $t \geq 0$. When $G$ is a cycle, a $b$-polytope $P \subseteq \mathcal{R}_{\mathcal{G}}$ (which has no vertices) is \textbf{Delzant} if the polytope $\Delta_Z \subseteq \mathfrak{t}_w^*$ is Delzant.
\end{definition}

The left polytope in Figure \ref{fig:polytopes} is not Delzant -- the Delzant condition is not satisfied at the vertex at the top of the picture in the center column of lattice points. The right polytope in Figure \ref{fig:polytopes} is Delzant. Given a Delzant $b$-polytope $P$, the intersection of $P$ with a component of $\mathcal{Z}_{\mathcal{G}}$ is a Delzant polytope in $\mathfrak{t}_w^*$. By the properties of $b$-polytopes, this Delzant polytope does not depend on which component of $\mathcal{Z}_{\mathcal{G}}$ is chosen.

\begin{definition} Given a $b$-polytope $P$, the \textbf{extremal polytope} $\Delta_P$ is the Delzant polytope in $\mathfrak{t}_w^*$ given by the intersection of $P$ with any connected component of $\mathcal{Z}_{\mathcal{G}}$.
\end{definition}

\noindent Before proving the Delzant theorem in our context, we require the following:

\begin{proposition}\label{prop:technicalpropfordelzant}
Let $(X_{\Delta}, \omega_{\Delta}, \mathbb{T}^{n-1}, \mu_{\Delta}:X_{\Delta} \rightarrow \Delta)$ be a (classic) compact connected symplectic toric manifold, and $a < b \in \mathbb{R}$. Consider the non-compact symplectic toric manifold
\[
(M = (a, b) \times \SSS^1 \times X_{\Delta}, \omega_M = dy \wedge d\theta + \omega_{\Delta}, \SSS^1 \times \mathbb{T}^{n-1}, (y, \mu_\Delta): (a, b) \times \Delta)
\]
where $y$ and $\theta$ are the standard coordinates on $(a, b)$ and $\SSS^1$ respectively. This symplectic toric manifold has the property that any symplectic vector field which is tangent to the fibers of the moment map is Hamiltonian.
\end{proposition}
\begin{proof} Choose any $y_0 \in (a, b), x_0 \in X_{\Delta}$, and consider the loop
$$\gamma: \SSS^1 \rightarrow (a,b) \times \SSS^1 \times X_{\Delta}, \text{\,\,\,\,\,\,}t \mapsto (y_0, t, x_0).$$
Integration of a $1$-form on $\gamma$ represents an element of $H^1(M)^*$ which pairs nontrivially with $[d\theta]$ and hence is itself nontrivial. By the K\"unneth formula,
\[
H^1(M) \cong (H^0((a, b) \times \SSS^1) \otimes H^1(X_{\Delta})) \oplus (H^1((a, b) \times \SSS^1) \otimes H^0(X_{\Delta}))
\]
which is one-dimensional due to the fact that the cohomology of a compact symplectic toric manifold is supported in even degrees. Therefore, a closed 1-form on $M$ is exact precisely if its integral along $\gamma$ is zero.

Let $v$ be a symplectic vector field on $M$ tangent to the fibers of the moment map. Because the fibers of the moment map are isotropic and because the image of $\gamma$ is contained in a single such fiber, it follows that $\omega_M(v, \gamma_*(\partial / \partial t)) = 0$ at all points in the image of $\gamma$. Therefore, the integral of $\iota_{v}\omega$ along $\gamma$ vanishes, so $\iota_v\omega$ is exact and therefore $v$ is Hamiltonian.
\end{proof}

\begin{theorem}\label{thm:bDelzant}
The map
\begin{equation}\label{eqn:bDelzantbijection1}
\left\{ \begin{array}{c} $b$-symplectic \ toric  \ manifolds\\ (M, Z, \omega, \mu:M \rightarrow \mathcal{R}_\mathcal{G}) \end{array} \right\} \rightarrow \left\{\begin{array}{c} \textrm{Delzant b-polytopes}\\ \textrm{in} \ \mathcal{R}_{\mathcal{G}} \end{array}\right\}
\end{equation}
that sends a $b$-symplectic toric manifold to the image of its moment map is a bijection, where $b$-symplectic toric manifolds are considered up to equivariant $b$-symplectomorphisms that preserve the moment map.
\end{theorem}

\begin{proof}

To prove surjectivity let $P$ be a Delzant $b$-polytope, and construct the (classic) symplectic toric manifold $(X_Z, \omega_Z, \mu_{Z}: X_Z \rightarrow \mathfrak{t}_Z^*)$ associated with the extremal polytope $\Delta_P$. Pick some $X \in \mathfrak{t}$ that pairs nontrivially with the distinguished vector in the definition of ${^b}\mathfrak{t}^*$, which induces an identification $\mathcal{R}_{\mathcal{G}} \cong \mathcal{R}_{\mathcal{G}}^{\prime} \times \mathfrak{t}_Z^*$, where $\mathcal{R}_{\mathcal{G}}^{\prime}$ is a one-dimensional $b$ moment codomain. If $G$ is a cycle, $P$ is the product of all of $\mathcal{R}_{\mathcal{G}}^{\prime}$ with $\Delta_P$; let $(\TT^2, Z_{T}, \omega_T, \mu_T: T^2 \rightarrow \mathcal{R}_{\mathcal{G}}^{\prime})$ be a $b$-symplectic toric manifold having all of $\mathcal{R}_{\mathcal{G}}^{\prime}$
as its moment map image. Then
\[
(\TT^2 \times X_Z, \omega_T \times \omega_Z, (\mu_T, \mu_Z))
\]
is a $b$-symplectic toric manifold having $P$ as the image of its moment map. If $G$ is a line, let $I$ be a closed interval in $\mathcal{R}_{\mathcal{G}}^{\prime}$ large enough that $I \times \Delta_{P} \supseteq P$. Let $(\SSS^2, Z_{S}, \omega_S, \mu_S: \SSS^2 \rightarrow \mathcal{R}_{\mathcal{G}}^{\prime})$ be a $b$-symplectic toric manifold having $I \subseteq \mathcal{R}_{\mathcal{G}}^{\prime}$ as its moment map image. Then
\[
(\SSS^2 \times X_Z, \omega_S \times \omega_Z, (\mu_S, \mu_Z))
\]
is a $b$-symplectic toric manifold having $I \times \Delta_{P}$ as the image of its moment map. By performing a sequence of symplectic cuts, we arrive at a $b$-symplectic toric manifold having $P$ as its moment map image.

The proof of injectivity is inspired by the proof of Proposition 6.4 in \cite{lertol}.
Let $(M, Z, \omega, \mu)$ and $(M', Z', \omega', \mu')$ be two $b$-symplectic toric manifolds having the same moment map image. Pick a lattice element $X \in \mathfrak{t}$ representing a generator of $\mathfrak{t}^* / \mathfrak{t}^*_w$. For each component $\{e\} \times \mathfrak{t}_Z^*$ of $\mathcal{Z}_{\mathcal{G}}$, by Proposition \ref{prop:local_model} there is an $\varepsilon > 0$ such that there is an equivariant isomorphism $\varphi_{Z_e}: \mu^{-1}(P_{Z_e}) \rightarrow \mu'^{-1}(P_{Z_e})$, where
\[
P_{Z_e} = \{-\varepsilon \leq y_{X, e} \leq \varepsilon\} \times \Delta_Z \subseteq P
\]
Similarly, for sufficiently large $N$, there exists an equivariant isomorphism $\varphi_{W_v}: \mu^{-1}(P_{W_v}) \rightarrow \mu'^{-1}(P_{W_v})$, where
$$P_{W_v} = ((-N, N)\times \mathfrak{t}_Z^*) \cap P \subseteq \{v\} \times \mathfrak{t}^* \subseteq \mathcal{R}_{\mathcal{G}},$$
we pick $N$ such that the open sets $\{P_{W_v}\} \cup \{P_{Z_e}\}$ cover $P$ as in Figure \ref{fig:subsetsofpolytopes}.

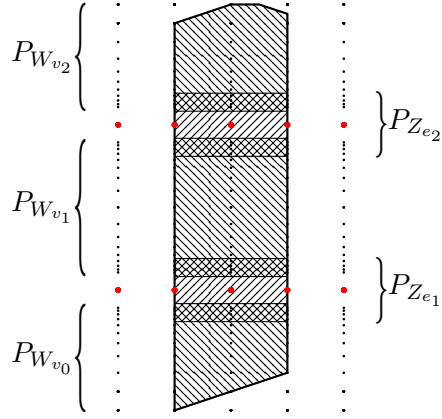
\begin{figure}[ht]
\centering
\begin{tikzpicture}
\pgfmathsetmacro{\startxone}{0}
\pgfmathsetmacro{\startxt}{6} \pgfmathsetmacro{\verlinebottom}{0}
\pgfmathsetmacro{\epsi}{0.12}
\pgfmathsetmacro{\awayfromz}{0.3}

\draw [thin, pattern = north west lines] (\startxone - 0.75, \verlinebottom + 2.2 + 0.1 + 1.25) -- (\startxone, \verlinebottom + 2.2 + 0.1 + 1.5)  -- (\startxone + 0.375, \verlinebottom + 2.2 + 0.1 + 1.5) -- (\startxone + 0.75, \verlinebottom + 2.2 + 0.1 + 1.375) -- (\startxone + 0.75, \verlinebottom + 2.2 + \awayfromz - \epsi) -- (\startxone - 0.75, \verlinebottom + 2.2 + \awayfromz - \epsi) -- cycle;

\draw [thin, pattern = north east lines] (\startxone + 0.75, \verlinebottom + 2.2 + \awayfromz + \epsi) -- (\startxone - 0.75, \verlinebottom + 2.2 + \awayfromz + \epsi) -- (\startxone - 0.75, \verlinebottom + 2.2 - \awayfromz - \epsi) -- (\startxone + 0.75, \verlinebottom + 2.2 - \awayfromz - \epsi) -- cycle;

\draw [thin, pattern = north west lines] (\startxone - 0.75, \verlinebottom + 2.2 - \awayfromz + \epsi) -- (\startxone + 0.75, \verlinebottom + 2.2 - \awayfromz + \epsi) -- (\startxone + 0.75, \verlinebottom + \awayfromz - \epsi) -- (\startxone - 0.75, \verlinebottom + \awayfromz - \epsi) -- cycle;

\draw [thin, pattern = north east lines] (\startxone + 0.75, \verlinebottom + \awayfromz + \epsi) -- (\startxone - 0.75, \verlinebottom + \awayfromz + \epsi) -- (\startxone - 0.75, \verlinebottom - \awayfromz - \epsi) -- (\startxone + 0.75, \verlinebottom - \awayfromz - \epsi) -- cycle;

\draw [thin, pattern = north west lines] (\startxone - 0.75, \verlinebottom - \awayfromz + \epsi) -- (\startxone + 0.75, \verlinebottom - \awayfromz + \epsi) -- (\startxone + 0.75, \verlinebottom - 0.1 - 1) 	-- (\startxone - 0.75, \verlinebottom - 0.1 - 1.5) -- cycle;

\draw [thick] (\startxone - 0.75, \verlinebottom + 2.2 + 0.1 + 1.25) -- (\startxone, \verlinebottom + 2.2 + 0.1 + 1.5)  -- (\startxone + 0.375, \verlinebottom + 2.2 + 0.1 + 1.5) -- (\startxone + 0.75, \verlinebottom + 2.2 + 0.1 + 1.375) -- (\startxone + 0.75, \verlinebottom - 0.1 - 1) 	-- (\startxone - 0.75, \verlinebottom - 0.1 - 1.5) -- cycle;

\node[yscale=2.3] at (1.5 + 0.5 ,\verlinebottom) {$\}$};
\node[scale=1] at (1.5 + 0.5 + 0.4,\verlinebottom) {$\ P_{Z_{e_1}}$};

\node[yscale=2.3] at (1.5 + 0.5 ,\verlinebottom + 2.2) {$\}$};
\node[scale=1] at (1.5 + 0.5 + 0.4,\verlinebottom + 2.2) {$\ P_{Z_{e_2}}$};

\node[yscale=3.75] at (-1.5 - 0.5 ,\verlinebottom - 0.9) {$\{$};
\node[scale=1] at (-1.5 - 0.5 - 0.5,\verlinebottom - 0.9) {$P_{W_{v_0}}$};

\node[yscale=4.75] at (-1.5 - 0.5 ,\verlinebottom + 1.1) {$\{$};
\node[scale=1] at (-1.5 - 0.5 - 0.5,\verlinebottom + 1.1) {$P_{W_{v_1}}$};

\node[yscale=3.75] at (-1.5 - 0.5 ,\verlinebottom + 2.2 + 0.9) {$\{$};
\node[scale=1] at (-1.5 - 0.5 - 0.5,\verlinebottom + 2.2  + 0.9) {$P_{W_{v_2}}$};

\foreach \horizshift in {-2, -1, 0, 1, 2}{
\foreach \paramt in {-2, -1.75, ..., 0.01}	
	{
	\pgfmathsetmacro{\height}{1 * exp(\paramt)}
	\pgfmathsetmacro{\squeezedhorizshift}{0.75 * \horizshift}
	\draw[fill = black] (\startxone + \squeezedhorizshift, \verlinebottom + \height + 0.1) circle(.1mm);
	\draw[fill = black] (\startxone + \squeezedhorizshift, \verlinebottom - \height - 0.1) circle(.1mm);
	\draw[fill = black] (\startxone + \squeezedhorizshift, \verlinebottom + 2.2 + \height + 0.1) circle(.1mm);
	\draw[fill = black] (\startxone + \squeezedhorizshift, \verlinebottom + 2.2 - \height - 0.1) circle(.1mm);

	\foreach \extratail in {0.25, 0.5}
		{
		\pgfmathsetmacro{\extrat}{\extratail * 1}
		\draw[fill = black] (\startxone + \squeezedhorizshift, \verlinebottom + 2.2 + 0.1 + 1 + \extrat) circle(.1mm);
		\draw[fill = black] (\startxone + \squeezedhorizshift, \verlinebottom -0.1 - 1 - \extrat) circle(.1mm);				
		}	
	\draw[red, fill = red] (\startxone + \squeezedhorizshift, \verlinebottom) circle(.3mm);
	\draw[red, fill = red] (\startxone + \squeezedhorizshift, \verlinebottom + 2.2) circle(.3mm);
	}
	}
	
\end{tikzpicture}
\caption{The subsets $P_{Z_{e_i}}$ and $P_{W_{v_i}}$ of a Delzant $b$-polytope.}
\label{fig:subsetsofpolytopes}
\end{figure}

If the equivariant isomorphisms $\varphi_{Z_{e_i}}$ and $\varphi_{W_{v_j}}$ agreed on $U_{ij} := \mu^{-1}(P_{W_{v_i}} \cap P_{Z_{e_j}})$ for all $i, j$, we could glue these isomorphisms together and the proof of injectivity would be complete. Therefore, it suffices to show for every $U_{ij}$ that there is an equivariant automorphism $\psi_{W_{v_i}}$ of $\mu^{-1}(P_{W_{v_i}})$ such that \[
\restr{\varphi_{W_{v_i}}\circ \psi_{W_{v_i}}}{U_{ij}} = \restr{\varphi_{Z_{e_j}}}{U_{ij}} \ \ \ \textrm{and} \ \ \ \restr{\varphi_{W_{v_i}}\circ \psi_{W_{v_i}}}{U_{ik}} = \restr{\varphi_{W_{v_j}}}{U_{ik}} \
\]
for $k \neq j$. Then by replacing $\varphi_{W_{v_i}}$ with $\varphi_{W_{v_i}} \circ \psi_{W_{v_i}}$, the isomorphisms $\varphi_{Z_{e_i}}$ and $\varphi_{W_{v_j}}$ can be glued. Repeating this process for each $U_{ij}$ gives the desired global equivariant isomorphism.

Let $\phi$ be the automorphism of  $U_{ij}$ given by $\varphi_{W_{v_i}}^{-1} \circ \varphi_{Z_{e_j}}$. We must extend this automorphism to an automorphism of $\mu^{-1}(P_{W_{v_i}})$ which is the identity outside an arbitrarily small neighborhood of $U_{ij}$. Notice that $\phi$ is a $\TT$-equivariant symplectic diffeomorphism that preserves orbits. Therefore, by Theorem 3.1 in \cite{haesal}, there exists a smooth $\TT$-invariant map $h: U_{ij} \rightarrow \TT^{n}$ such that $\phi(x) = h(x) \cdot x$. By the $\TT$-invariance of $h$ and the contractibility of $\mu(U_{ij}) = P_{W_{v_i}} \cap P_{Z_{e_j}}$, there is a map $\theta: U_{ij} \rightarrow \mathfrak{t}$ such that $\textrm{exp} \circ \theta = h$. Define the vector field $X_{\theta}$ to be $X_{\theta}(x) = \restr{\frac{d}{ds}}{s = 0} \textrm{exp}(s\theta(x))\cdot x$. Observe that $X_{\theta}$ is a symplectic vector field whose time one flow is the symplectomorphism $\phi$. By Proposition \ref{prop:technicalpropfordelzant}, the vector field is Hamiltonian. Pick an $\hat{f}$ such that $d\hat{f} = \iota_{X_{\theta}}\omega$. Extend $\hat{f}$ to be a function $f$ on all of $\mu^{-1}(P_{W_{v_i}})$ that is locally constant outside a small neighborhood of $U_{ij}$. Then the time-1 flow of the Hamiltonian vector field corresponding to $f$ will be the desired symplectic automorphism of $\mu^{-1}(P_{W_{v_i}})$.
\end{proof}

\begin{remark}
The proof of surjectivity in Theorem \ref{thm:bDelzant} is unlike the proof of surjectivity in the classic Delzant theorem, since we do not construct the $b$-symplectic manifold globally through a symplectic cut in some large $\mathbb{C}^d$. However, we suspect that such a construction is possible by considering an appropriate extension of $\mathbb{C}^d$ similar to the extension of $\mathfrak{t}^*$ to $\mathcal{R}_{\mathcal{G}}$.
\end{remark}

\begin{remark}\label{rmk:classification in dim 2 and 4}
The moment image of a $2n$-dimensional $b$-symplectic toric manifold is represented by an $n$-dimensional polytope $P$, and the corresponding extremal polytope $\Delta_P$ is an $(n-1)$-dimensional Delzant polytope.

For $n=1$, the extremal polytope is a point, and therefore a $b$-symplectic toric surface is equivariantly $b$-symplectomorphic to either a $b$-symplectic torus $\TT^2$ or a $b$-symplectic sphere, as stated earlier in Theorem \ref{thm:surfaces}.

For $n=2$, the extremal polytope is a line segment, corresponding to a symplectic toric sphere. As a consequence, a $b$-symplectic toric $4$-manifold is equivariantly $b$-symplectomorphic to either a product $\TT^2\times\SSS^2$ of a $b$-symplectic torus with a symplectic sphere, or to a manifold obtained from the product $\SSS^2\times\SSS^2$ of two spheres, one $b$-symplectic and the other symplectic, by a series of symplectic cuts which avoid $Z$. In particular, $\CC P^2\# \overline{\CC P^2}$ can be obtained from a $b$-symplectic $\SSS^2\times\SSS^2$ with connected $Z$ via symplectic cutting and therefore can be endowed with a $b$-symplectic toric structure. Because $Z$ was connected (in fact, it would suffice for $Z$ to have an odd number of connected components), there will be fixed points in both the portion of the manifold with positive orientation and in the one with negative orientation. Blowing up these fixed points (each such blow up destroys one fixed point and creates two new ones with the same orientation) corresponds to performing connect sum with either $\CC P^2$ or $\overline{\CC P^2}$, according to the orientation. Therefore, any $m\CC P^2\# n\overline{\CC P^2}$, with $m,n\geq1$ can be endowed with $b$-symplectic toric structures (compare with \cite[Fig.\ 1, Cor.\ 5.2]{cavalcanti}). Observe that $\CC P^2\# n\overline{\CC P^2}$, with $n\geq1$ admits both symplectic and $b$-symplectic toric structures.

\end{remark}

\begin{remark}
The surprising consequence of the Delzant classification theorem for $b$-symplectic toric manifolds is that the existence of a $b$-symplectic toric structure is highly restrictive. Indeed, every $b$-symplectic toric manifold is either the product of a $b$-symplectic $\mathbb{T}^2$ with a classic symplectic toric manifold, or it can be obtained from the product of a $b$-symplectic $\mathbb{S}^2$ with a classic symplectic toric manifold by a sequence of symplectic cuts performed at the north and south ``polar caps'', away from the exceptional hypersurface $Z$.

\end{remark}

\begin{remark}\label{rmk:anotherway}

The procedure of reduction by stages provides a different perspective on Delzant $b$-polytopes and on how to reconstruct toric $b$-symplectic manifolds from them. It can be shown that the reduction of $M$ by the action of the subtorus $\mathbb T^{n-1}_Z$ at a regular point is a $2$-dimensional $b$-symplectic manifold with a natural Hamiltonian $\mathbb{S}^1$-action. Theorem \ref{thm:surfaces} exactly classifies toric $b$-surfaces by data that can be encoded in their one-dimensional moment $b$-polytope. Therefore, a Delzant $b$-polytope can be viewed as a fibration over a classic Delzant polytope $\Delta_Z \subseteq \mathfrak{t}_Z^*$ whose fibers are one-dimensional $b$-polytopes, and the corresponding $b$-symplectic toric manifold can be reconstructed as a fibration over the classic symplectic toric manifold corresponding to $\Delta_Z$, with fibers the $b$-symplectic toric surfaces determined by the one-dimensional $b$-polytopes.
\end{remark}


\end{document}